\documentclass[notitlepage, onecolumn]{article}

\usepackage{graphicx}
\usepackage{amsmath, amsthm, mathrsfs}
\usepackage{amssymb}
\usepackage[dvips]{epsfig}
\usepackage[noadjust]{cite}
\usepackage{setspace}
\usepackage[hmarginratio=1:1,top=32mm,columnsep=20pt]{geometry}
\usepackage{hyperref}

\newtheorem{theorem}{Theorem}
\newtheorem{remark}{Remark}
\newtheorem{lemma}{Lemma}
\newtheorem{assumption}{Assumption}

\newtheorem{definition}{Definition}

\newtheorem{corollary}{Corollary}

\newcommand{\bs}{\boldsymbol}

\newcommand{\mbf}{\mathbf}
\newcommand{\be}{\begin{equation}}
\newcommand{\ee}{\end{equation}}

\date{}
\title{Containment Control of Second-order Multi-agent Systems Under Directed Graphs and Communication Constraints}

\author{A. Abdessameud, I. G. Polushin, and A. Tayebi
\thanks{This work was supported by the Natural Sciences and Engineering Research Council of Canada (NSERC).}
\thanks{The authors are with the Department of Electrical and Computer Engineering, University of Western Ontario, London, Ontario, Canada. The third author is also with the Department of Electrical Engineering, Lakehead University, Thunder Bay, Ontario, Canada.
{\tt\small aabdess@uwo.ca, ipolushi@uwo.ca, tayebi@ieee.org} }%
}

\begin{document}

\maketitle

\begin{abstract}
            The distributed coordination problem of multi-agent systems is addressed under the assumption of intermittent discrete-time information exchange with time-varying (possibly unbounded) delays. Specifically, we consider the containment control problem of second-order multi-agent systems with multiple dynamic leaders under a directed interconnection graph topology. First, we present distributed control algorithms for double integrator dynamics in the full and partial state feedback cases. Thereafter, we propose a method to extend our results to second-order systems with locally Lipschitz nonlinear dynamics. We show that, under the same information exchange constraints, our approach can be applied to solve similar coordination problems for other types of complex second-order multi-agent systems, such as harmonic oscillators. In all cases, our control objectives are achieved under some conditions that can be realized independently from the interconnection topology and from the characteristics of the communication process. The effectiveness of the proposed control schemes is illustrated through some examples and numerical simulations. \vspace{-0.1 in}
\end{abstract}

\section{Introduction}

            The distributed coordination problem of dynamical multi-agent systems has received a growing interest during the last decade due to the broad range of applications involving multiple vehicle systems. The main idea behind distributed coordination is to ensure a collective behavior using local interaction between the agents. This interaction, in the form of information exchange, is generally performed using communication between agents according to their interconnection graph topology. Examples of collective behaviors include consensus, synchronization, flocking, and formation maintenance \cite{qu2009cooperative, Ren:Cao:book}. While multi-agent systems with general linear dynamics have been widely considered in the literature (see, for instance, \cite{scardovi2009synchronization, li2010consensus, li2015designing} and references therein), coordinated control algorithms for multi-agent systems with double integrator dynamics have been actively studied with a particular interest to second-order consensus problems \cite{ren2007information, ren2008distributed, abdessameud2010consensus, abdessameud2013consensus}, cooperative tracking with a single leader \cite{ren2008distributed,  zhao2013distributed}, and the containment control problem with multiple stationary or dynamic leaders \cite{cao2011distributed, li2012distributed, liu2012necessary}. Despite their simple dynamical model, it has been shown that coordinating a team of such agents is a difficult problem especially if one considers restrictions on the interconnection graph between agents or some constraints related to the dynamics of agents such as the lack of velocity measurements and/or input saturations.
            In addition, basic concepts from the above mentioned results have been successfully applied to harmonic oscillators \cite{ren2008synchronization, su2009synchronization}, and second-order nonlinear multi-agent systems with general unknown and globally Lipschitz nonlinearities \cite{liu2013consensus, yu2013distributed, song2013m} or with special models such as Euler-Lagrange systems \cite{Chen:Lewis:2011, Mei:Ren:2011, Mei:Ren:2012, meng2014leader}, attitude dynamics \cite{abdess:TAC:2009, dimos:2009, cai2014leader} and under-actuated unmanned vehicles \cite{lawton2003decentralized, aabdess:tay:book}.

            In this paper, we consider the containment control problem of linear and nonlinear second-order multi-agent systems under a directed interconnection graph. The main control objective is to drive the positions of a group of agents called followers to the convex hull spanned by the positions of another group of dynamic agents called leaders. In contrast to the above mentioned results where ideal communication between agents is assumed, our interest in this work is to design distributed control algorithms in the presence of communication constraints generally imposed in practical situations. Actually, the communication process can be subject to unknown delays and information losses, which may affect the system performance or even destroy the system's stability. Also, communication between agents may not be performed continuously in time, but intermittently at some discontinuous time-intervals or at specified instants of time. This situation can be simply imposed on the communication process to save energy/communication costs in mobile agents, or induced by environmental constraints, such as communication obstacles, and temporary sensor/communication-link failure.

            Recently, various control algorithms have been developed for multi-agent systems in the presence of some of the above mentioned communication constraints. The authors in \cite{yu2010some, sun2009consensus2, lin2011multi, liu2014containment, yan2014containment}, for example, consider consensus problems (including the containment control problem in \cite{liu2014containment, yan2014containment}) for double integrators in the presence of communication delays under different assumptions on the interconnection graph. In these papers, all agents reach some agreement on their final positions with a common constant final velocity provided that self delays are implemented and topology-dependent conditions are satisfied. Note that using self-delays requires perfect measurements of the communication delays. The work in \cite{meng2011leaderless} presents leader-follower schemes for double integrators in the presence of constant communication delays without using self delays. However, only uniform boundedness of the relative position errors between agents is reached under some conditions on the delays and the interconnection topology between agents. In \cite{zhou2014consensus}, state- and output-feedback algorithms that achieve consensus for linear multi-agent systems with more general dynamics have been proposed in the leaderless case. In the latter work, the effects of the communication delays have been compensated using some prediction of the current states of neighboring agents obtained using their old (delayed) states and the constant communication delay, which is assumed to be perfectly known. However, achieving consensus on a dynamic final state is still a challenging problem in the presence of communication delays, especially, with possible information losses that prevent measurements of the delays. The problem becomes more difficult in the case where only partial state measurements are available for feedback. It should be noted that such prediction is not required in the case where  agents are driven to a stationary position. This can be seen in the literature in consensus algorithms for linear multi-agent systems \cite{Munz:CDC, abdessameud2011unified, abdessameud2012hetero} and some classes of nonlinear systems in the presence of constant communication delays \cite{Spong:Chopra:2007, chung:2009, munz:2011, Nuno11, abdess:IFAC:2011, Wang:2013} and time-varying communication delays \cite{Erdong:2008, abdess:VTOL:2011, Abdess:attitude:TAC:2012, Nuno13, Abdessameud:Polushin:Tayebi:2013:ieeetac}
            under undirected and/or directed topologies.  Also, all the above delay-robust results for linear and nonlinear multi-agent systems share the common assumption of continuous-time communication between agents.

            In the case of intermittent communication, the authors in \cite{sun2009consensus} consider first-order multi-agents and suggest to hold, using a zero-order-hold system, the relative positions of interacting agents each time this information is received. In the presence of sufficiently small constant communication delays and bounded packet dropout, the proposed discontinuous algorithm in \cite{sun2009consensus} achieves consensus provided that self-delays are implemented and the non-zero update period of the zero-order-hold system is small.
            A similar approach has been applied for double integrators in \cite{gao2010asynchronous, gao2010consensus}, where asynchronous and synchronous updates of the zero-order-hold systems have been addressed, respectively. Using a different approach, a switching algorithm achieving second-order consensus has been proposed in \cite{Wen:Duan:2012} in cases where communication between agents is lost during small intervals of time. The latter result has been extended to multi-agent systems with general linear dynamics \cite{Wen:Ren:2013} and globally Lipschitz nonlinear dynamics \cite{Wen:L2:2012}, where it has been shown that consensus can still be achieved under some conditions on the communication rates and interaction topology between agents. However, communication delays have not been considered in \cite{gao2010asynchronous, gao2010consensus, Wen:Duan:2012, Wen:Ren:2013, Wen:L2:2012}, and it is not clear whether the methods in these papers are still valid in the case of discrete-time communication. More recently, based on the small-gain approach presented in \cite{Abdessameud:Polushin:Tayebi:2013:ieeetac}, a solution to the synchronization problem of a class of nonlinear second-order systems has been presented in  \cite{abdessameud2015synchronization} assuming intermittent and delayed discrete-time communication between agents. However, the result in \cite{abdessameud2015synchronization} 
            cannot be extended in a straightforward manner to the case of multiple leaders with time-varying trajectories.

            The main contribution of this paper consists in providing distributed containment control algorithms for second-order multi-agent systems under a directed interconnection topology and in the presence of the above mentioned communication constraints. More precisely, we consider the case where the communication between agents is discrete in time, intermittent, asynchronous, and subject to non-uniform and unknown irregular communication delays and possible packets dropouts.
            The combination of these communication constraints implies that each agent may receive information (with delays) from other agents in the network only at the endpoints of some intervals of time, that we refer to as blackout intervals.
            Based on the small gain theorem used in our earlier work \cite{Abdessameud:Polushin:Tayebi:2013:ieeetac}, we present an approach for the design and analysis of distributed control algorithms, achieving containment control with multiple dynamic leaders, under mild assumptions on the directed interconnection topology provided that the communication blackout intervals are finite. Using this approach, we present distributed containment control algorithms for linear second-order multi-agents modeled by double integrators with and without measurements of the velocities of the followers. In this case, the dynamic leaders are assumed to be moving according to some uniformly bounded and vanishing acceleration. Next, we present a systematic method to solve the containment control problem of a network of non-identical agents with nonlinear dynamics under the same assumption on the motion of the leaders. The latter result is unifying in the sense that it can be applied to a wide class of second-order systems with locally Lipschitz nonlinearities. Then, we relax our assumptions on the leaders' final states and show that our approach can be applied to solve the containment control problem of a class of linear second-order oscillators, including harmonic oscillators, under the same communication constraints.

            To the best of our knowledge, the containment control problem of linear and nonlinear second-order systems
            has never been addressed
            under similar assumptions on the interconnection and communication between agents. As compared to the relevant literature (see, for instance, \cite{yu2010some, sun2009consensus2, lin2011multi, liu2014containment, yan2014containment, sun2009consensus, meng2011leaderless, zhou2014consensus, gao2010asynchronous, gao2010consensus, Wen:Duan:2012, Wen:Ren:2013, Wen:L2:2012}), the proposed approach in this paper achieves our control objectives by taking into account the above mentioned communication constraints simultaneously without imposing conservative assumptions on the interconnection topology between agents. As compared to \cite{abdessameud2015synchronization} that handles similar communication constraints, we address in this paper the more challenging containment control problem with multiple leaders with time-varying trajectories. Our approach in this work is more general and takes into account several considerations related to the dynamics of the leaders and the followers, the availability of states measurements, as well as the above communication constraints, within the same framework. Moreover, containment control, in each of the cases in this study, is achieved under simple design conditions that do not depend neither on the interconnection topology between agents nor on the maximal communication blackout interval which can be unknown and can take large values. The effectiveness of the proposed containment control algorithms is shown through several numerical examples.

\section{Background and Problem Formulation}\label{SecProbStat}
        Throughout the paper, we use $|x|$ to denote the Euclidean norm of a vector $x\in\mathbb{R}^q$ with $\mathbb{R}^q$ being the $q$-dimensional Euclidean space. We denote with $I_q$ and $0_{q}$, respectively, the $q$-dimensional identity matrix and the zero matrix of dimension $q\times q$, and we let $0_{q\times p}\in\mathbb{R}^{q\times p}$ denote the matrix of all zeros. We also use $\mbf{1}_q\in\mathbb{R}^q$ to denote the vector of all ones. The spectral radius of a square matrix $A$ is denoted by $\rho(A)$.
        The Kronecker product of matrices $A$ and $B$ is denoted by $A\otimes B$. The limit $\lim_{t\to +\infty} x(t) = c$ is denoted by $x(t)\to c$.

\subsection{Graph theory background}

        Consider a system composed of $n$  agents that are interconnected in the sense that some information can be transmitted between agents using communication channels. The interconnection topology between agents is modeled by a weighted directed graph $\mathcal{G}=(\mathcal{N},\mathcal{E},\mathcal{A})$ where each agent is represented by a node and $\mathcal{N}:=\{1,\ldots, n\}$ is the set of all nodes. The set $\mathcal{E}\in\mathcal{N}\times\mathcal{N}$ contains ordered pairs of nodes, called edges, and $\mathcal{A}=[a_{ij}]\in\mathbb{R}^{n\times n}$ is the weighted adjacency matrix. An edge $(j, i)\in\mathcal{E}$ is represented by a directed link (an arrow) from node $j$ to node $i$, and indicates that agent $i$ can obtain information from agent $j$ but not {\it vice versa}; in this case, we say that $j$ and $i$ are neighbors (even though the link between them is directed). A finite ordered sequence of distinct edges of $\mathcal{G}$ with the form $(j,l_1),(l_1,l_2), \ldots,  (l_q, i)$ is called a directed path from $j$ to $i$.
        A directed graph $\mathcal{G}$ is said to contain a spanning tree if there exists at least one node that has a directed path to all the other nodes in $\mathcal{G}$.
        The weighted adjacency matrix is defined such that $a_{ii}:= 0$, $a_{ij}>0$ if  $(j,i)\in\mathcal{E}$, and  $a_{ij}=0$ if $(j,i)\notin \mathcal{E}$.  The Laplacian matrix $\mathcal{L}:=[l_{ij}]\in\mathbb{R}^{n\times n}$ associated to the directed graph $\mathcal{G}$ is defined such that: $l_{ii}=\sum_{j=1}^n a_{ij}$, and $l_{ij}=-a_{ij}$ for $i\neq j$. That is, $\mathcal{L}:= \mathcal{D} - \mathcal{A}$, where $\mathcal{D}$, called the in-degree matrix, is a diagonal matrix with the $(i,i)-$th entry being $l_{ii}$.

        In this paper, we consider the case where there exist $m$ followers and $n-m$ leaders ($m < n$). Without loss of generality, we let $\mathfrak{F}:=\{1,\ldots,m\}$ and $\mathfrak{L}:=\{m+1, \dots, n\}$ denote the follower set and the leader set, respectively, so that $\mathfrak{L} = \mathcal{N}\setminus\mathfrak{F}$. Here, an agent $i$ is called a leader, or $i\in\mathfrak{L}$, if $(j,i)\notin \mathcal{E}$ for each $j\in\mathcal{N}$; a leader node does not receive information from any other node in $\mathcal{G}$. An agent $i$ is called a follower, or $i\in\mathfrak{F}$, if $(j,i)\in \mathcal{E}$ for at least one $j\in\mathcal{N}$. Accordingly, matrices $\mathcal{D}$, $\mathcal{A}$, and $\mathcal{L}$ associated with $\mathcal{G}$ take the form
         \begin{eqnarray}\label{D}
         \mathcal{D} &=&  \begin{bmatrix} \mathcal{D}_1 & 0_{m\times \bar{m}}\\0_{\bar{m}\times m}&0_{\bar{m}\times \bar{m}}\end{bmatrix}, \quad \mathcal{A} = \begin{bmatrix} \mathcal{A}_1&\mathcal{A}_2\\ 0_{\bar{m}\times m}&0_{\bar{m}\times \bar{m}}\end{bmatrix}\\
         \label{laplacian}
          \mathcal{L}&=&\left[\begin{array}{ll} \mathcal{L}_1&\mathcal{L}_2\\ 0_{\bar{m}\times m}&0_{\bar{m}\times \bar{m}}\end{array}\right],
         \end{eqnarray}
        with $\bar{m} :=(n-m)$, $\mathcal{L}_1 = \mathcal{D}_1 - \mathcal{A}_1$ and $\mathcal{L}_2 = -\mathcal{A}_2$.
        \begin{assumption}\label{assumption_graph}
        For each node $i\in\mathfrak{F}$, there exists at least one node $j\in\mathfrak{L}$ such that a directed path from $j$ to $i$ exists in $\mathcal{G}$. That is, for each follower, there exists at least one leader having a directed path to the follower.
        \end{assumption}
        Consider the following definition and Lemma used in the subsequent analysis.
        \begin{definition}\cite{qu2009cooperative} Let $\mathcal{Z}_n \subset \mathbb{R}^{n\times n}$ denote the set of all square matrices of dimension $n$ with non-positive off-diagonal entries. A matrix $A \in\mathbb{R}^{n\times n}$ is said to be a nonsingular M-matrix if $A \in \mathcal{Z}_n$ and all eigenvalues of $A$ have positive real parts. Further, a nonsingular M-matrix $A \in \mathcal{Z}_n$ can be written as $A = s I - M$ for some scalar $s>0$ and matrix $M \geq 0$ such that $s > \rho(M)$.\label{definition_M_matrix}
        \end{definition}
        \begin{lemma}\label{lemma_Mei}\cite[Lemma 2.3]{Mei:Ren:2012}
         Consider $\mathcal{L}$ defined in \eqref{laplacian}. Under Assumption~\ref{assumption_graph}, the matrix $\mathcal{L}_1$  is a nonsingular M-matrix, each entry of $ -\mathcal{L}^{-1}_1 \mathcal{L}_2$ is nonnegative, and all row sums of $-\mathcal{L}_1^{-1} \mathcal{L}_2$ are equal to one.
        \end{lemma}

\subsection{Communication process}\label{sec:comm}

            In this work, we consider the case where communication between agents is discrete in time, intermittent, and may be subject to unknown irregular delays and information losses. Specifically, for each pair $(j,i)\in\mathcal{E}$, there exists a strictly increasing unbounded sequence $\mathscr{S}_{ij}\subseteq \{0,1, 2, 3, \ldots \}$ such that the $j$-th agent is allowed to send data to the $i$-th agent at instants $t_{k_{ij}}$, $k_{ij} \in \mathscr{S}_{ij}$. This information exchange is subject to a sequence of communication delays $\left(\tau_{k_{ij}}\right)_{k_{ij} \in \mathscr{S}_{ij}}$, with $\tau_{k_{ij}}\in[0,+\infty]$, meaning that the information sent by the $j$-th agent at $t_{k_{ij}}$ can be received by the $i$-th vehicle at instant $t_{k_{ij}}+\tau_{k_{ij}}$. In particular, the case where $\tau_{k_{ij}}=+\infty$ implies that the corresponding data has been lost during transmission or has never been sent. The following assumptions are imposed on the communication process.

            \begin{assumption}\label{AssumptionCommunicationBlackouts}
                    For each pair $(j,i)\in\mathcal{E}$, there exists a strictly increasing infinite subsequence $\bar{\mathscr{S}}_{ij} = \{k_{ij}^{(1)}, k_{ij}^{(2)}, \ldots \}\subseteq \mathscr{S}_{ij}$  such that:
                    $t_{k_{ij}^{(l+1)}} +\tau_{k_{ij}^{(l+1)}} - t_{k_{ij}^{(l)}}\le T^*$, for $l=1, 2, \ldots$ and for some $T^*>0$.
            \end{assumption}

            \begin{assumption}\label{assump_common_sampling}
                For all $(j,i)\in\mathcal{E}$, the sequence of time instants $t_{k_{ij}}$ satisfies: $t_{k_{ij}}= k_{ij} T$, $k_{ij}\in\mathscr{S}_{ij}$, with $T>0$ being a sampling period available to all agents.
            \end{assumption}
            Assumption~\ref{AssumptionCommunicationBlackouts} essentially states that, for each pair $(j,i)\in\mathcal{E}$, the information sent by the $j$-th agent at the instants $t_{k_{ij}}$, $k_{ij} \in \bar{\mathscr{S}}_{ij}\subseteq \mathscr{S}_{ij}$, are successfully received by the $i$-th agent with the corresponding communication delays. In addition, for each pair $(j,i)\in\mathcal{E}$, the maximum length of communication blackout intervals between the $j$-th and $i$-th agents does not exceed an arbitrary (not necessarily known) bound $T^*$.  Assumption~\ref{assump_common_sampling} is common in sampled-data communication protocols and specifies a common sampling rate to the transmitted data in the network. Note, however, that $\mathscr{S}_{ij}$ and $\bar{\mathscr{S}}_{ij}$ are defined for each edge in $\mathcal{E}$ which indicates that the information exchange described above is asynchronous.

\subsection{Problem statement}

            Consider the $n$ systems interconnected according to a directed graph $\mathcal{G}$, and suppose that the communication process between agents is as described in Section \ref{sec:comm}. The dynamics of agents will be described in the subsequent sections. For $i\in\mathcal{N}$, let $p_i\in\mathbb{R}^N$ denote the position-like state of each agent. Also, let $\mbf{p}_F$ and $\mbf{p}_L$ be the column stack vectors of $p_i$, $i\in\mathfrak{F}$, and $p_i$, $i\in\mathfrak{L}$, respectively\footnote{Throughout the paper, we use notation $\mbf{x}_F$ and $\mbf{x}_L$ to denote the column stack vectors of $x_i$ for $i\in\mathfrak{F}$ and $x_i$ for $i\in\mathfrak{L}$, respectively.}, and let $\mathcal{S}_{L}(t):=\{ p_{m+1}(t), \ldots, p_n(t)\}$. The objective of this work is to design  distributed control schemes that solve the containment control problem, where the positions of the follower agents are to be driven to the convex hull spanned by the positions of the leaders. Formally, it is required that
            \begin{equation}\label{objective1}
            d(p_i(t),  \mathcal{C}[\mathcal{S}_{L}(t)])  \to 0,\qquad \mbox{for}~i\in\mathfrak{F},
            \end{equation}
            where for a point $x$ and a set $M$, $d(x, M)$ denotes the distance between $x$ and $M$, {\it i.e.,} $d(x, M):= \inf_{y\in M} |x-y|$, and $\mathcal{C}[X]:= \{\sum_{j=1}^p \alpha_j x_j ~ | ~x_j \in X, ~\alpha_j \geq 0, ~ \sum_{j=1}^p \alpha_j =1 \}$ denotes the convex hull of the set $X:=\{ x_1, \ldots, x_p\}$ \cite{Mei:Ren:2012}. Under Assumption~\ref{assumption_graph}, the result of Lemma~\ref{lemma_Mei} implies that $-(\mathcal{L}_1^{-1}\mathcal{L}_2\otimes I_N)\mbf{p}_L$ is within the convex hull spanned by the leaders \cite{Mei:Ren:2012}. Therefore, objective  \eqref{objective1} is reached if one guarantees that $\mbf{p}_F(t) + (\mathcal{L}_1^{-1}\mathcal{L}_2\otimes I_N)\mbf{p}_L(t) \to 0$.

            To achieve the above objective, we let the vector $\mbf{p}_j^{(i)}(k_{ij})$ denote the information that can be transmitted from agent $j$ to agent $i$ at instant $t_{k_{ij}}$, for each $(j,i)\in\mathcal{E}$ and $k_{ij}\in \mathscr{S}_{ij}$. In particular, $\mbf{p}_j^{(i)}(k_{ij}):=[\chi_j(t_{k_{ij}}), k_{ij}]$, where $\chi_j(t_{k_{ij}})$ is a vector, to be defined later, that contains some of the $j$-th agent's states measured at instant $t_{k_{ij}}$, and $k_{ij}$ is the sequence number of transmission instants $t_{k_{ij}}$. Accordingly, for each pair $(j,i)\in\mathcal{E}$ and each time instant $t\geq 0$, let $k_{ij}^{\mathrm{m}}(t)$ denote the largest integer number such that $\mbf{p}_j^{(i)}(k_{ij}^{\mathrm{m}}(t))=[\chi_{j}(t_{k_{ij}^{\mathrm{m}}(t)}), k_{ij}^{\mathrm{m}}(t)]$ is the most recent information of agent $j$ that is already delivered to agent $i$ at $t$.
            It should be noted that the number $k_{ij}^{\mathrm{m}}(t)$ can be obtained by a simple comparison of the received sequence numbers.
            
\vspace{-0.1 in}
\section{Technical Lemma}\label{sec:lemma}
            Before we proceed, we present in this section a unifying result that will simplify the forthcoming analysis. Consider a system of $n$-agents interconnected according to $\mathcal{G}$ and governed by the following dynamics
            \begin{eqnarray}
                       \label{filter_leader_lemma}
                       \dot{\eta}_i &=& \Psi_i, \quad ~i\in\mathfrak{L},\\
                       \label{filter_follower_lemma}
                       \dot{\eta}_i &=& -k_{\eta_i}(\eta_i - \delta_i) + \Phi_{i,1}, \quad~i\in\mathfrak{F},
                       \end{eqnarray}
            where  $k_{\eta_i}>0$, $\eta_i\in\mathcal{R}^N$, $i\in\mathcal{N}$, is the position-like state of the $i$-th agent,  the sets $\mathfrak{L}$ and $\mathfrak{F}$ are defined as above, and $\delta_i\in\mathbb{R}^N$, $i\in\mathfrak{F}$, is the output of the following system
            \begin{eqnarray}
            \label{filter:lemma} \dot{\zeta}_i &=& (H_i\otimes I_N) \zeta_i + (B_i\otimes I_N) \varepsilon_i + \Phi_{i,2}\\
             \label{output_filter_lemma}\delta_i &=& \alpha (C_i \otimes I_N) \zeta_i + (1-\alpha)\varepsilon_i \\
             \label{input_filter_lemma}
             \varepsilon_i &=&   \textstyle\frac{1}{\kappa_i}\sum_{j=1}^n a_{ij} \eta_{j}(t_{k_{ij}^{\mathrm{m}}(t)})
             \end{eqnarray}
             for $i\in\mathfrak{F}$, where $\zeta_i\in\mathbb{R}^{\sigma_i N}$, $\sigma_i\in \mathbb{N}$, the vector $\varepsilon_i\in\mathbb{R}^N$ is considered as the input of system \eqref{filter:lemma}-\eqref{output_filter_lemma} with $a_{ij}$ being the $(i,j)-$th entry of the adjacency matrix $\mathcal{A}$ associated to $\mathcal{G}$ and $\kappa_i := \sum_{j=1}^n a_{ij}$ for $i\in\mathfrak{F}$, and the matrices $H_i\in\mathbb{R}^{\sigma_i\times \sigma_i}$, $B_i\in\mathbb{R}^{\sigma_i\times 1}$, and $C_i\in\mathbb{R}^{1\times \sigma_i}$ are given by
             \begin{equation}\label{gains_filter}\begin{array}{l}
                 \setlength{\arraycolsep}{2pt}
                        \renewcommand{\arraystretch}{0.9}
             H_i = \begin{bmatrix}-h_{i,1}&h_{i,1}&0&\ldots&0&0\\0&-h_{i,2}&h_{i,2}&\ldots&0&0\\ \vdots&\vdots&\ddots&\ldots&\vdots&\vdots \\ &&&\ddots&&\\0&0&\ldots&&-h_{i,\sigma_i-1}&h_{i,\sigma_i-1}\\ 0&0&&\ldots&0&-h_{i,\sigma_i}\end{bmatrix},\\  B_i = \begin{bmatrix} 0& \ldots &0& h_{i,\sigma_i}\end{bmatrix}^{\top}, ~
             C_i = \begin{bmatrix} 1&0&\ldots&0\end{bmatrix}\end{array}
             \end{equation}
             with $h_{i,1},~ \ldots, h_{i,\sigma_i}>0$ and $i\in\mathfrak{F}$. The vectors $\Phi_{i,1}$, $\Phi_{i,2}$, $i\in\mathfrak{F}$, and $\Psi_i$, $i\in\mathfrak{L}$, are considered as perturbation terms.

             In \eqref{filter_follower_lemma}-\eqref{input_filter_lemma}, the $i$-th follower, $i\in\mathfrak{F}$, uses the received position-like states from its corresponding neighbors that are transmitted using the communication process described in Section~\ref{sec:comm}. In particular, the vector $\eta_{j}(t_{k_{ij}^{\mathrm{m}}(t)})$ is the most recent position-like state of agent $j$ that is available to agent $i$ at instant $t$. Also, the parameter $\alpha$ is either $1$ or $0$; in particular, $\alpha = 0$ implies that the dynamic system \eqref{filter:lemma} is not implemented.
             \begin{lemma}\label{theorem_unified}
                Consider the multi-agent system \eqref{filter_leader_lemma}-\eqref{input_filter_lemma} and suppose Assumption~\ref{assumption_graph} and Assumption~\ref{AssumptionCommunicationBlackouts} hold. Then, with $\alpha=1$ or $\alpha=0$, the following holds for arbitrary initial conditions:
                \begin{itemize}
                \item [i)] If the vectors $\Phi_{i,1}$, $\Phi_{i,2}$, for $i\in\mathfrak{F}$, and $\Psi_i$, for $i\in\mathfrak{L}$, are uniformly bounded, then $\dot{\bs{\zeta}}_F$, $\dot{\bs{\eta}}_F$ and $\bs{\eta}_F + (\mathcal{L}_1^{-1}\mathcal{L}_2\otimes I_N) \bs{\eta}_L$ are uniformly bounded.
                    \item [ii)] If, in addition, $\Phi_{i,1}(t)\to 0$, $\Phi_{i,2}(t)\to 0$, $i\in\mathfrak{F}$, $\Psi_i(t)\to 0$, $i\in\mathfrak{L}$, then $\dot{\bs{\zeta}}_F(t)\to 0$,  $\dot{\bs{\eta}}_F(t)\to 0$, and  $\bs{\eta}_F(t) + (\mathcal{L}_1^{-1}\mathcal{L}_2\otimes I_N) \bs{\eta}_L(t) \to 0$.
                        \end{itemize}
                Furthermore, if the conditions in item $i)$ hold and the perturbation terms $\Phi_{i,1}$, $\Phi_{i,2}$, for $i\in\mathfrak{F}$, and $\Psi_i$, for $i\in\mathfrak{L}$, do not converge to zero, then the effects of these perturbation terms can be reduced with a choice of $k_{\eta_i}$ and the entries of $H_i$ provided that $T^*$ is sufficiently small and/or $\bs{\eta}_L(t)$ is slowly varying.
               \end{lemma}
               \begin{proof} See Appendix~\ref{section_proofs}.
               \end{proof}

\vspace{-0.1 in}
\section{Containment Control of Double-Integrators}\label{sec:linear:systems}
       In this section, we consider the multi-agent system governed by
       \begin{equation}\label{model}
        \dot{p}_i = v_i,\qquad
        \dot{v}_i = \Gamma_i,\quad i\in\mathcal{N}, 
        \end{equation}
        where $\Gamma_i$, for $i\in\mathfrak{F}$, is the input vector for each follower, and $\Gamma_i$ for $i\in\mathfrak{L}$ is a function chosen such that each leader evolves with some bounded acceleration and all the leaders converge asymptotically to a common steady state velocity. Specifically, we consider the following assumption.
        \begin{assumption}\label{asum:leader:motion}
        For all $i\in\mathfrak{L}$, $\Gamma_i$ is a uniformly  bounded function such that the state $v_i$ in \eqref{model} is uniformly bounded and $\dot{v}_i(t) \to 0$, $v_i(t)\to v_d\in\mathbb{R}^{N}$, for arbitrary initial conditions.
        \end{assumption}

        Also, we assume that $\chi_j(t_{k_{ij}}) := [p_j(t_{k_{ij}}), \hat{v}_j(t_{k_{ij}})]$, where $p_j$ is the position of agent $j$ for $j\in\mathcal{N}$, $\hat{v}_j \equiv v_j$, for $j\in\mathfrak{L}$, and $\hat{v}_j$, for $j\in\mathfrak{F}$, denotes a velocity estimate obtained by the $j$-th follower according to an algorithm described below. Using this information exchange, we consider the following control input for each follower in \eqref{model}
        \begin{eqnarray}\label{input:dynamic:1}
             \Gamma_i &=& - k_{d_i} (v_i-\hat{v}_i) - k_{p_i} (p_i - \psi_i) \\
\label{filter_observer_dynamic}
        \dot{\hat{v}}_i &=& -L_{p_i} \Big(\hat{v}_i - \frac{1}{\kappa_i}\sum_{j=1}^n a_{ij}\hat{v}_j(t_{k_{ij}^{\mathrm{m}}(t)})\Big)
        \end{eqnarray}
        for $i\in\mathfrak{F}$, where the vector $\psi_i$ is given by
        \begin{equation}\label{psi_input_dynamic1}
          \mbf{\psi}_i = \frac{1}{\kappa_i}\sum_{j=1}^n a_{ij} \underbrace{\Big(p_j(t_{k_{ij}^{\mathrm{m}}(t)})+ \hat{v}_{j}(t_{k_{ij}^{\mathrm{m}}(t)})\cdot (t - k_{ij}^{\mathrm{m}}(t) T)\Big)}_{\vartheta_{ij}(t)}
          \end{equation}
        and $k_{p_i}$, $k_{d_i}$, $L_{p_i}$, $i\in\mathfrak{F}$, are strictly positive scalar gains, $a_{ij}$ is the $(i,j)$-th element of $\mathcal{A}$, and $\kappa_i := \sum_{j=1}^n a_{ij}$ for $i\in\mathfrak{F}$ satisfies $\kappa_i \neq 0$ in view of Assumption~\ref{assumption_graph}. The term $\vartheta_{ij}$ in \eqref{psi_input_dynamic1} can be regarded as an estimate of the current position of the $j$-th agent, for $(j,i)\in\mathcal{E}$. This term is based on the most recent information of the $j$-th agent available to the $i$-th agent at instant $t$, and depends on the common sampling period $T$, which is available to all agents as per Assumption~\ref{assump_common_sampling}.

        It can be noticed that the control law \eqref{input:dynamic:1}-\eqref{psi_input_dynamic1} might be discontinuous in the presence of the irregularities of the received information due to the communication constraints. To
        ensure a continuous-time control action, the received information can be moved one integrator away from the control input by letting the vector $\psi_i$ in \eqref{input:dynamic:1} be the solution of the following dynamic system
          \begin{eqnarray}
            \label{dot_psi_input_dynamic}\dot{\psi}_i &=& - k_{\psi_i} \psi_i +  \frac{k_{\psi_i}}{\kappa_i}\sum_{j=1}^n a_{ij} \vartheta_{ij} + \hat{v}_i, \quad i\in\mathfrak{F},
            \end{eqnarray}
        with $k_{\psi_i}>0$, and  $\hat{v}_i$, $\vartheta_{ij}$ given in 
        \eqref{filter_observer_dynamic}-\eqref{psi_input_dynamic1}.

         In the case where the velocity vectors $v_i$ for $i\in\mathfrak{F}$ are not available for feedback, we consider the following control algorithm
            \begin{eqnarray}\label{input:dynamic:withoutvelocity}
             \Gamma_i &=& -k_{d_i}(\phi_i + p_i - \hat{v}_i)- k_{p_i} (p_i - \psi_i)\\
             \label{dot_phi_input_withoutvelocity}
             \dot{\phi}_i &=& -(\phi_i + p_i) + \Gamma_i
            \end{eqnarray}
            for $i\in\mathfrak{F}$, where $\hat{v}_i$, $i\in\mathcal{N}$, and $\psi_i$, $i\in\mathfrak{F}$, are defined in \eqref{filter_observer_dynamic} and \eqref{psi_input_dynamic1} (or \eqref{dot_psi_input_dynamic}), respectively.

            The signals $\psi_i$ and $\hat{v}_i$ in \eqref{input:dynamic:1} (and in \eqref{input:dynamic:withoutvelocity}) can be considered as a reference position and a reference velocity for each follower agent. The dynamic system \eqref{filter_observer_dynamic}, acting as a distributed observer, is introduced such that all followers reach some agreement on their velocity estimates in the presence of communication constraints. For this, the velocity estimates $\hat{v}_i$, $i\in\mathcal{N}$, are transmitted between agents instead of the actual velocity vectors. In contrast, the reference position $\psi_i$ in \eqref{psi_input_dynamic1} (or in \eqref{dot_psi_input_dynamic}) depends on the received positions of the corresponding neighboring agents. 
            The signal $\phi_i\in\mathbb{R}^N$ is used in \eqref{input:dynamic:withoutvelocity} to cope for the lack of measurements of the velocity signals $v_i$, $i \in\mathfrak{F}$. The control input $\Gamma_i$ in \eqref{input:dynamic:1} (respectively, in \eqref{input:dynamic:withoutvelocity}-\eqref{dot_phi_input_withoutvelocity}) is designed such that each follower tracks its reference velocity and reference position (respectively, without measurements of the followers' velocities) in the presence of the communication constraints described in Section~\ref{sec:comm}. This can be shown using the small gain framework (Lemma~\ref{theorem_unified}) as stated in the following result.

    \begin{theorem}\label{theorem_nonstationary}
                Consider the network of $n$ systems described by~\eqref{model} and suppose Assumptions~\ref{assumption_graph}-\ref{asum:leader:motion} hold. For each $i\in\mathfrak{F}$, consider the control input \eqref{input:dynamic:1}-\eqref{filter_observer_dynamic} with \eqref{psi_input_dynamic1} or \eqref{dot_psi_input_dynamic}. Pick the gains $k_{p_i}$ and $k_{d_i}$ such that the roots of $x^2+k_{d_i} x+k_{p_i}=0$ are real. Then, the vectors $\dot{\bs{v}}_F$, $\bs{v}_F$, $\mbf{p}_F + (\mathcal{L}_1^{-1}\mathcal{L}_2\otimes I_N)\mbf{p}_L$ are uniformly bounded, $v_i(t)\to v_d$, for $i\in\mathfrak{F}$, and $\mbf{p}_F(t) + (\mathcal{L}_1^{-1}\mathcal{L}_2\otimes I_N)\mbf{p}_L(t)\to 0$ for arbitrary initial conditions.\\
                Furthermore, if $\Gamma_i(t)$, $i\in\mathfrak{L}$, does not converge to zero, then, the effects of $\Gamma_i$, $i\in\mathfrak{L}$, on the error signal $\mbf{p}_F + (\mathcal{L}_1^{-1}\mathcal{L}_2\otimes I_N)\mbf{p}_L$ can be reduced with a choice of $L_{p_i}$, $k_{p_i}$ and $k_{d_i}$ provided that $T^*$ is small and/or the signal $v_i(t)$, $i\in\mathfrak{L}$, is slowly varying. \\
                The same above results hold when using the control input \eqref{input:dynamic:withoutvelocity}-\eqref{dot_phi_input_withoutvelocity} with \eqref{filter_observer_dynamic} and \eqref{psi_input_dynamic1} or \eqref{dot_psi_input_dynamic}. \hfill$\blacksquare$
    \end{theorem}
  \begin{proof}
        Consider the dynamics of the distributed observer \eqref{filter_observer_dynamic} with $\hat{v}_i \equiv v_i$ for $i\in\mathfrak{L}$, which can be rewritten as in  \eqref{filter_leader_lemma}-\eqref{input_filter_lemma} with $\alpha = 0$. Then, one can show, using Lemma~\ref{theorem_unified} with Assumption~\ref{AssumptionCommunicationBlackouts} and Assumption~\ref{asum:leader:motion}, that         $\dot{\hat{\bs{v}}}_F$, $\hat{\bs{v}}_F + (\mathcal{L}_1^{-1}\mathcal{L}_2\otimes I_N)\bs{v}_L$ are uniformly bounded and asymptotically converge to zero for arbitrary $T^*>0$. Then, $\hat{v}_i$, $i\in\mathfrak{F}$, is uniformly bounded  and $\hat{v}_i(t) \to v_d$, $i\in\mathfrak{F}$, since $\hat{v}_i(t)\to v_d$, $i\in\mathfrak{L}$, and the row sums of $-\mathcal{L}_1^{-1}\mathcal{L}_2$ are equal to one (by Lemma~\ref{lemma_Mei}). In addition, if $\dot{v}_i(t)$, $i\in\mathfrak{L}$, does not converge to zero, then one can show, using Lemma~\ref{theorem_unified}, that the effects of a non-vanishing $\dot{v}_i(t)$, $i\in\mathfrak{L}$, on the error signal $\hat{\bs{v}}_F + (\mathcal{L}_1^{-1}\mathcal{L}_2\otimes I_N)\bs{v}_L$ can be reduced with a choice of $L_{p_i}$ provided that $T^*$ is small and/or $\hat{v}_i(t)$, $i\in\mathfrak{L}$, is a slowly varying signal.

        Consider system \eqref{model} with \eqref{input:dynamic:1}, which can be rewritten as
        \begin{eqnarray}\textstyle
        \label{dot_v_pf_thm1}\dot{v}_i &=& - k_{d_i} \textstyle(v_i-\hat{v}_i) - k_{p_i} (p_i - \nu_i)+\bs{\epsilon}_i\\
        \label{nu_pf_thm1}  \nu_i &=& \textstyle\bar{\alpha}\psi_i + \textstyle\frac{(1-\bar{\alpha})}{\kappa_i}\sum_{j=1}^n a_{ij} \vartheta_{ij}\\
        \label{dot_psi_pf_thm1} \dot{\psi}_i &=&\textstyle - k_{\psi_i} \psi_i +  \textstyle\frac{k_{\psi_i}}{\kappa_i}\sum_{j=1}^n a_{ij} \vartheta_{ij} + \hat{v}_i
          \end{eqnarray}
         for $i\in\mathfrak{F}$, where $\bar{\alpha}=0$ or $\bar{\alpha}=1$ correspond, respectively, to using \eqref{psi_input_dynamic1} or \eqref{dot_psi_input_dynamic}, and $\bs{\epsilon}_i \equiv 0$.  For analysis purposes, suppose that $\bs{\epsilon}_i$ is uniformly bounded and $\bs{\epsilon}_i(t)\to 0$.

        Let $-\lambda_i < 0$ be one of the real roots of $x^2+k_i^d x+k_i^p=0$, for $i\in\mathfrak{F}$. Consider also the new variable $\xi_i = \frac{1}{\lambda_i} (v_i-\hat{v}_i) + p_i$, $i\in\mathfrak{F}$, which, in view of \eqref{model} and \eqref{dot_v_pf_thm1}, satisfies
        \begin{eqnarray}
              \dot{\xi}_i &=& \textstyle\frac{1}{\lambda_i}\Big( -k_{d_i} (v_i-\hat{v}_i) - k_{p_i}(p_i - \nu_i)\Big) + v_i - \frac{1}{\lambda_i}(\dot{\hat{v}}_i- \bs{\epsilon}_i)\nonumber\\
              \label{dot_xi_pf_thm1}&=& \textstyle-\frac{k_{p_i}}{\lambda_i}(\xi_i - \nu_i) + \hat{v}_i - \frac{1}{\lambda_i}(\dot{\hat{v}}_i- \bs{\epsilon}_i), \qquad i\in\mathfrak{F},
        \end{eqnarray}
            where the last equality is obtained using the relations $\dot{p}_i = v_i = \lambda_i(\xi_i - p_i) + \hat{v}_i$ and $k_{d_i}=\lambda_i + \frac{k_{p_i}}{\lambda_i}$, which hold from the definition of $\xi_i$ and $\lambda_i$, respectively.

            Now, define $\bar{p}_i = (p_i - v_{d} t)$ for $i\in\mathcal{N}$, and $\bar{\xi}_i = (\xi_i - v_{d} t)$, $\bar{\psi}_i = (\psi_i - v_{d} t)$, $\bar{\nu}_i = (\nu_i - v_d t)$, for $i\in\mathfrak{F}$. Then, using \eqref{model} and \eqref{dot_v_pf_thm1}-\eqref{dot_xi_pf_thm1}, one can show that
                \begin{equation}\label{filter_closed_full_linear_1:1}\begin{array}{lcll}
                \dot{\bar{p}}_i &=& v_i - v_{d}, &i\in\mathfrak{L}\\
                \dot{\bar{p}}_i &=& -\lambda_i(\bar{p}_i - \bar{\xi}_i) + \hat{v}_i - v_d, & i\in\mathfrak{F},
                \end{array}
                \end{equation}
                with\vspace{-0.1 in}
                \begin{eqnarray}
                \dot{\bar{\xi}}_i &=& \displaystyle\frac{-k_{p_i}}{\lambda_i}(\bar{\xi}_i - \bar{\nu}_i) + \hat{v}_i - v_d  - \frac{1}{\lambda_i}(\dot{\hat{v}}_i- \bs{\epsilon}_i)\\
                \bar{\nu}_i &=& \bar{\alpha}\bar{\psi}_i + (1-\bar{\alpha})\Big(\frac{1}{\kappa_i}\sum_{j=1}^n a_{ij} \bar{p}_{j}(t_{k_{ij}^{\mathrm{m}}(t)}) + \Theta_{i,1}\Big)\\
                \dot{\bar{\psi}}_i &=& -k_{\psi_i}\Big(\bar{\psi}_i - \frac{1}{\kappa}_i\sum_{j=1}^n a_{ij} \bar{p}_{j}(t_{k_{ij}^{\mathrm{m}}(t)}) \Big) + \Theta_{i,2}\end{eqnarray}
                and
                {\small{\begin{equation}\label{Upsilon_thm_dynamic}\begin{array}{l}
                \Theta_{i,1} = \textstyle\frac{1}{\kappa_i} \sum_{j=1}^n a_{ij} \left( (\hat{v}_{j}(t_{k_{ij}^{\mathrm{m}}(t)}) - v_{d})\cdot( t - k_{ij}^{\mathrm{m}}(t) T)\right)\\
                \Theta_{i,2} = k_{\psi_i}\Theta_{i,1} + \hat{v}_i - v_d
                \end{array}\end{equation}}}
                  for $i\in\mathfrak{F}$.
            It is then clear that for $\bar{\alpha} = 0$, the closed loop dynamics \eqref{filter_closed_full_linear_1:1}-\eqref{Upsilon_thm_dynamic} can be written in the form \eqref{filter_leader_lemma}-\eqref{gains_filter} with $\alpha = 1$ and $\sigma_i = 1$, for $i\in\mathfrak{F}$. Also, for $\bar{\alpha} = 1$, system \eqref{filter_closed_full_linear_1:1}-\eqref{Upsilon_thm_dynamic} is equivalent to  \eqref{filter_leader_lemma}-\eqref{gains_filter} with $\alpha = 1$ and $\sigma_i = 2$, for $i\in\mathfrak{F}$. Therefore, the result of the theorem can be shown by verifying the conditions in items $i)$ and $ii)$ of Lemma~\ref{theorem_unified}. By Assumption \ref{asum:leader:motion}, we know that $\dot{\bar{p}}_i$ is uniformly bounded and $\dot{\bar{p}}_i(t)\to 0$ for $i\in\mathfrak{L}$. Also, we have shown that $\dot{\hat{v}}_i$, $\hat{v}_i$ are uniformly bounded, $\dot{\hat{v}}_i(t)\to 0$, and $\hat{v}_i(t) - v_d \to 0$ for all $i\in\mathfrak{F}$. This, with the fact that $t_{k_{ij}^{\mathrm{m}}(t)} = k_{ij}^{\mathrm{m}}(t)T \to +\infty$ and $(t - k_{ij}^{\mathrm{m}}(t) T)\leq T^*$, in view of Assumption~\ref{AssumptionCommunicationBlackouts}, lead one to conclude that $\Theta_{i,1}$, $\Theta_{i,2}$, in \eqref{Upsilon_thm_dynamic}, are uniformly bounded and $\Theta_{i,1}(t)\to 0$, $\Theta_{i,2}(t)\to 0$, for $i\in\mathfrak{F}$.

            Invoking Lemma~\ref{theorem_unified}, we can show that $\dot{\bar{\bs{\xi}}}_F$, $\dot{\bar{\mbf{p}}}_F$, $\bar{\mbf{p}}_F +(\mathcal{L}_1^{-1}\mathcal{L}_2\otimes I_N)\bar{\mbf{p}}_L$ are uniformly bounded and $\dot{\bar{\mbf{p}}}_F(t)\to 0$, $\bar{\mbf{p}}_F(t) + (\mathcal{L}_1^{-1}\mathcal{L}_2\otimes I_N)\bar{\mbf{p}}_L(t)\to 0$, as $t\to+\infty$ for arbitrary initial conditions and for arbitrary $T^*>0$. Consequently, $\dot{v}_i$, $(v_i - v_d)$ are uniformly bounded and $\dot{v}_i(t)\to 0$, $(v_i(t) - v_d)\to 0$, $i\in\mathfrak{F}$. Also, since all row sums of $-\mathcal{L}_1^{-1}\mathcal{L}_2$ are equal to one (see Lemma~\ref{lemma_Mei}), we can  show that $\bar{\mbf{p}}_F +(\mathcal{L}_1^{-1}\mathcal{L}_2\otimes I_N)\bar{\mbf{p}}_L = {\mbf{p}}_F +(\mathcal{L}_1^{-1}\mathcal{L}_2\otimes I_N){\mbf{p}}_L$, which leads to the conclusion  $\mbf{p}_F + (\mathcal{L}_1^{-1}\mathcal{L}_2\otimes I_N)\mbf{p}_L$ is uniformly bounded and $\mbf{p}_F(t) +(\mathcal{L}_1^{-1}\mathcal{L}_2\otimes I_N)\mbf{p}_L(t)\to 0$.

            In the case where $\dot{v}_i$, $i\in\mathfrak{L}$, does not converge to zero, we know that $\hat{v}_i - v_d$, $\dot{\hat{v}}_i$, $\Theta_{i,1}$, and $\Theta_{i,2}$, $i\in\mathfrak{F}$, do not converge to zero. As discussed above, the effects of a non-zero $\dot{v}_i$, $i\in\mathfrak{L}$, on $\hat{v}_i - v_d$ and  $\dot{\hat{v}}_i$, $i\in\mathfrak{F}$, can be reduced with a choice of $L_{p_i}$ for small $T^*$ and/or a slowly varying  $\hat{v}_i(t)$, $i\in\mathfrak{L}$. Similarly, one can deduce from Lemma~\ref{theorem_unified} that the effects of non-zero $\hat{v}_i - v_d$, $\dot{\hat{v}}_i$, $\Theta_{i,1}$, and $\Theta_{i,2}$, $i\in\mathfrak{F}$, on $\dot{\bar{\mbf{p}}}_F$ and  ${\mbf{p}}_F +(\mathcal{L}_1^{-1}\mathcal{L}_2\otimes I_N){\mbf{p}}_L$ can be further reduced using the gains $\lambda_{i}$ and $\frac{k_{p_i}}{\lambda_i}$ for small $T^*$ and/or a slowly varying $v_i(t)$, $i\in\mathfrak{L}$.

            Finally, in the case where $v_i$, $i\in\mathfrak{F}$, are not available for feedback, applying the control algorithm \eqref{input:dynamic:withoutvelocity}-\eqref{dot_phi_input_withoutvelocity} with \eqref{filter_observer_dynamic} and \eqref{psi_input_dynamic1} (or \eqref{dot_psi_input_dynamic}) to \eqref{model}, $i\in\mathfrak{F}$, leads to the closed loop dynamics
             \begin{eqnarray*}
             \dot{v}_i &=& - k_{d_i} (v_i-\hat{v}_i) - k_{p_i} (p_i - \nu_i) - k_{d_i}\tilde{\phi}_i\\
             \dot{\tilde{\phi}}_i &=& - \tilde{\phi}_i
             \end{eqnarray*}
             where $\tilde{\phi}_i := \phi_i + p_i - v_i$, $i\in\mathfrak{F}$, and $\nu_i$ is given in \eqref{nu_pf_thm1}-\eqref{dot_psi_pf_thm1}, for $i\in\mathfrak{F}$. Then, the results of the theorem can be shown following the same arguments as above by letting $\bs{\epsilon}_i := - k_{d_i}\tilde{\phi}_i$ in \eqref{dot_v_pf_thm1}.
 \end{proof}

             In the case where the steady state velocity of the leaders $v_d$ is available to all the followers, the distributed observer \eqref{filter_observer_dynamic} is not required and the following Corollary can be proved using similar steps as in the proof of Theorem~\ref{theorem_nonstationary}.

            \begin{corollary}\label{theorem_stationary}
                Consider the multi-agent system \eqref{model} and suppose Assumptions~\ref{assumption_graph}-\ref{asum:leader:motion} hold. For each $i\in\mathfrak{F}$, consider the control algorithm \eqref{input:dynamic:1} with \eqref{psi_input_dynamic1} (or \eqref{dot_psi_input_dynamic}) by setting $\hat{v}_i\equiv v_d$ for all $i\in\mathfrak{F}$. Also, suppose that the control gains are selected as in Theorem~\ref{theorem_nonstationary}. Then,
                $\dot{\bs{v}}_F$, $\bs{v}_F$, $\mbf{p}_F + (\mathcal{L}_1^{-1}\mathcal{L}_2\otimes I_N)\mbf{p}_L$  are uniformly bounded, $v_i(t)\to v_d$, for $i\in\mathfrak{F}$, and  $\mbf{p}_F(t) + (\mathcal{L}_1^{-1}\mathcal{L}_2\otimes I_N)\mbf{p}_L(t)\to 0$, for arbitrary initial conditions.  The same results hold when using the control input \eqref{input:dynamic:withoutvelocity}-\eqref{dot_phi_input_withoutvelocity} with \eqref{psi_input_dynamic1} (or \eqref{dot_psi_input_dynamic}).
            \end{corollary}

            \begin{remark} Note that, in the special case of multiple leaders with zero steady state velocity, {\it i.e.,} $v_d = 0$, or stationary leaders, {\it i.e.,} $v_i = 0$ for all $i\in\mathfrak{L}$, Assumption~\ref{assump_common_sampling} is not needed in Corollary~\ref{theorem_stationary}.
            \end{remark}

           The containment control problem for double integrators \eqref{model} has been addressed recently in \cite{liu2014containment} and \cite{yan2014containment} assuming continuous-time communication in the presence of uniform communication delays. With the assumptions that the leaders's velocities are constant and the communication delays are smooth and can be measured, objective~\ref{objective1} is achieved in \cite{liu2014containment}, using a state feedback algorithm, under some conditions directly related to the interconnection topology between agents, the upper bound of the communication delays, and the solutions of some LMIs. A similar result is obtained in \cite{liu2014containment, yan2014containment} in the case of constant communication delays. In contrast, the distributed control algorithms in Theorem~\ref{theorem_nonstationary} achieve objective~\eqref{objective1} under weaker assumptions on the communication between agents, that can be subject to unknown, non-uniform, and irregular communication delays with possible packets dropouts, and, in addition, remove the requirements of velocity measurements for the followers.
           Further, the above results solve the containment control problem for multi-agent system \eqref{model}, with multiple dynamic leaders satisfying Assumption~\ref{asum:leader:motion}, under a simple condition on the control gains. Interestingly, this condition is topology-free and does not depend on the ``unknown'' maximal blackout interval $T^*$ that may take arbitrarily large values.

\vspace{-0.1 in}
\section{Containment Control of Nonlinear Second-Order Multi-agents}\label{sec:nonlinear:systems}
            In this section, we consider the case where the dynamics of the followers are non-identical and  satisfy
            \begin{equation}\label{model_nl}
            \dot{p}_i(t) = v_i(t),\qquad
            \dot{v}_i(t) = F_i(p_i, v_i, \Gamma_i),\quad i\in\mathcal{N},
            \end{equation}
            where the functions $F_i: \mathbb{R}^N\times \mathbb{R}^N \times \mathbb{R}^N \to \mathbb{R}^N$, for $i\in\mathfrak{F}$, are assumed to be continuous and locally Lipschitz with respect to their arguments, and $F_i \equiv \Gamma_i$ for $i\in\mathfrak{L}$.

            Similarly to the previous section, we suppose that $\chi_j(t_{k_{ij}}) := [p_j(t_{k_{ij}}), \hat{v}_j(t_{k_{ij}})]$, where $\hat{v}_i$, $i\in\mathfrak{F}$, is a velocity estimate obtained using a distributed observer described below  and $\hat{v}_i \equiv v_i$, $i\in\mathfrak{L}$. Assuming that the velocities of the followers are available for feedback, we define a reference velocity signal for each follower as follows
            \begin{equation}\label{reference_velocity_nonlinear}
            v_{r_i} := -\lambda_{i}(p_i - \eta_{i,1}) + \hat{v}_i,
            \end{equation}
            where $\lambda_i>0$ and $\eta_i$, $\hat{v}_i$ satisfy
            \begin{equation}\label{filter_dynamic:nonlinear}\left\{\begin{array}{lcl}
            \dot{\eta}_{i,1} &=& -k_{p_i}(\eta_{i,1} - \eta_{i,2})+ \hat{v}_i\\
            \dot{\eta}_{i,2} &=& -k_{d_i} \Big(\eta_{i,2} - \frac{1}{\kappa_i}\sum_{j=1}^n a_{ij}\vartheta_{ij}\Big) + \hat{v}_i
            \end{array}\right.\end{equation}
            \begin{equation}\label{filter_observer:nonlinear}\left\{\begin{array}{lcl}
            \dot{\hat{v}}_i &=& -L_{p_i} (\hat{v}_i - \sigma_{i,1})\\
            \dot{\sigma}_{i,1} &=& -L_{d_i} \Big(\sigma_{i,1}-\frac{1}{\kappa_i}\sum_{j=1}^n a_{ij}\hat{v}_j(t_{k_{ij}^{\mathrm{m}}(t)})\Big)\end{array}\right.
            \end{equation}
            for $i\in\mathfrak{F}$, and where $k_{p_i}$, $k_{d_i}$, $L_{p_i}$, $L_{d_i}$, $i\in\mathfrak{F}$, are strictly positive scalar gains, $\vartheta_{ij}$ is given in \eqref{psi_input_dynamic1}, and $a_{ij}$, $\kappa_i$ are defined as above.

            The main idea behind the introduction of the reference velocity $v_{r_i}$ is different from that of the approach used in Section~\ref{sec:linear:systems}. This reference velocity is generated for each follower in \eqref{model_nl} using the states of the dynamic auxiliary systems  \eqref{filter_dynamic:nonlinear}-\eqref{filter_observer:nonlinear}. The structure of these auxiliary systems is motivated by the result of Lemma~\ref{theorem_unified} that is proved in Appendix~\ref{section_proofs} based on the small-gain framework.
            We will show in Theorem~\ref{theorem_nl} below that all followers coordinate their motion with respect to the leaders' trajectories if there exists a tracking control input $\Gamma_i$ in \eqref{model_nl}, $i\in\mathfrak{F}$, such that each follower tracks its corresponding reference velocity. 

            \begin{theorem}\label{theorem_nl}
                Consider the multi-agent system \eqref{model_nl} and suppose that Assumptions~\ref{assumption_graph}-\ref{asum:leader:motion} hold. For each $i\in\mathfrak{F}$, consider the reference velocity $v_{r_i}$ given in \eqref{reference_velocity_nonlinear}-\eqref{filter_observer:nonlinear}
                and suppose that there exists a control input $\Gamma_i$ that guarantees:
            \begin{itemize}
            \item [i)] The error vector $e_i := v_i - v_{r_i}$, $i\in\mathfrak{F}$,  is uniformly bounded.
            \item [ii)] If $\dot{v}_{r_i}$ and $v_{r_i}$ are uniformly bounded, then $e_i(t)\to 0$.
            \end{itemize}
             Then, $\bs{v}_F$, $\mbf{p}_F + (\mathcal{L}_1^{-1}\mathcal{L}_2\otimes I_N)\mbf{p}_L$ are uniformly bounded, $v_i(t)\to v_d$, for $i\in\mathfrak{F}$, and $\mbf{p}_F(t) + (\mathcal{L}_1^{-1}\mathcal{L}_2\otimes I_N)\mbf{p}_L(t)\to 0$ for arbitrary initial conditions.
             \hfill$\blacksquare$
            \end{theorem}

            \begin{proof}
           First, it can be verified that the distributed observer \eqref{filter_observer:nonlinear} with $\hat{v}_i = v_i$, $i\in\mathfrak{L}$, can be rewritten as \eqref{filter:lemma}-\eqref{gains_filter} with $\alpha = 1$ and $\sigma_i = 1$, for $i\in\mathfrak{F}$. Therefore, Lemma~\ref{theorem_unified} can be used to show that $\dot{\hat{\bs{v}}}_F$, $\hat{\bs{v}}_F+(\mathcal{L}_1^{-1}\mathcal{L}_2\otimes I_N)\bs{v}_L$ are uniformly bounded, $\dot{\hat{\bs{v}}}_F(t)\to 0$ and $\hat{\bs{v}}_F(t) + (\mathcal{L}_1^{-1}\mathcal{L}_2\otimes I_N)\bs{v}_L(t)\to 0$, in particular $\hat{v}_i(t)- v_d \to 0$ $i\in\mathfrak{F}$, under Assumption~\ref{assumption_graph}, Assumption~\ref{AssumptionCommunicationBlackouts} and Assumption~\ref{asum:leader:motion}.

                Next, define the following variables: $\bar{p}_i = (p_i - v_{d} t)$, for $i\in\mathcal{N}$, $\bar{\eta}_{i,1} = (\eta_{i,1} - v_{d} t)$, and $\bar{\eta}_{i,2} = (\eta_{i,2} - v_{d} t)$, for $i\in\mathfrak{F}$. The dynamics of these variables can be shown, after some computations using \eqref{reference_velocity_nonlinear}-\eqref{filter_dynamic:nonlinear}, to satisfy
                \begin{equation}\label{filter_closed_nl1}\begin{array}{lcl}
                \dot{\bar{p}}_i = v_i - v_d, && i\in\mathfrak{L},\\
                \dot{\bar{p}}_i = -\lambda_i (\bar{p}_i - \bar{\eta}_{i,1}) + \hat{v}_i - v_d + e_i, && i\in\mathfrak{F},\end{array}
                \end{equation}
                \begin{equation}\label{filter_closed_nl2}
                \begin{array}{l}
                \begin{bmatrix}\dot{\bar{\eta}}_{i,1}\\ \dot{\bar{\eta}}_{i,2}\end{bmatrix}=\begin{bmatrix} -k_{p_i}I_N&k_{p_i}I_N\\0_N &-k_{d_i}I_N\end{bmatrix} \begin{bmatrix} \bar{\eta}_{i,1} \\ \bar{\eta}_{i,2}\end{bmatrix}\\
                 ~~~~~~~~~~~+ \begin{bmatrix}\hat{v}_i - v_{d} \\ \frac{k_{d_i}}{\kappa_i}\textstyle\sum_{j=1}^n a_{ij} \bar{p}_{j}(t_{k_{ij}^{\mathrm{m}}(t)})\big) + \bar{\Theta}_{i}\end{bmatrix}
                \end{array}
                \end{equation}
            for $i\in\mathfrak{F}$, where $\bar{\Theta}_i:= k_{d_i}\Theta_{i,1} + \hat{v}_i - v_{d}$, with $\Theta_{i,1}$ being given in \eqref{Upsilon_thm_dynamic}. It can also be verified that \eqref{filter_closed_nl1}-\eqref{filter_closed_nl2} is equivalent to \eqref{filter:lemma}-\eqref{gains_filter} with $\alpha = 1$ and $\sigma_i = 2$, for $i\in\mathfrak{F}$. From item $i)$ in the theorem, we know that there exists an input $\Gamma_i$, $i\in\mathfrak{F}$, such that $e_i$ is uniformly bounded, $i\in\mathfrak{F}$. Also, Assumption~\ref{asum:leader:motion} implies that $\dot{\bar{p}}_i$ is uniformly bounded and $\dot{\bar{p}}_i(t)\to 0$, $i\in\mathfrak{L}$. Note that we have shown above that $\dot{\hat{v}}_i$, $\hat{v}_i-v_d$, $i\in\mathfrak{F}$, are uniformly bounded and asymptotically converge to zero. This with $t_{k_{ij}^{\mathrm{m}}(t)} \to +\infty$ and $(t - k_{ij}^{\mathrm{m}}(t) T)\leq T^*$, by Assumption~\ref{AssumptionCommunicationBlackouts}, one can deduce that $\bar{\Theta}_{i}$ is uniformly bounded and ${\bar\Theta}_{i}(t)\to 0$, $i\in\mathfrak{F}$. Invoking Lemma~\ref{theorem_unified}, we can show that  $\dot{\bar{\eta}}_{i,1}$, $\dot{\bar{\eta}}_{i,2}$, $\dot{\bar{p}}_i$, $i\in\mathfrak{F}$, and $\bar{\mbf{p}}_F + (\mathcal{L}_1^{-1}\mathcal{L}_2\otimes \mbf{I}_N)\bar{\mbf{p}}_L$ are uniformly bounded. Consequently, $v_{r_i}$ and $\dot{v}_{r_i}$, $i\in\mathfrak{F}$, are uniformly bounded. Then, item $ii)$ in the theorem leads us to conclude that $e_i(t) \to 0$, $i\in\mathfrak{F}$. Invoking Lemma~\ref{theorem_unified} again, we can show that $v_i(t)\to v_d$, $i\in\mathfrak{F}$, and $\mbf{p}_F(t) +(\mathcal{L}_1^{-1}\mathcal{L}_2\otimes I_N)\mbf{p}_L(t)\to 0$.
 \end{proof}
            \begin{remark}
            Similarly to Corollary~\ref{theorem_stationary}, the result in Theorem~\ref{theorem_nl} can be simplified in the case where the steady state velocity of the leaders $q_d$ is available to all the followers, or is null, by letting $\hat{v}_i = v_d$ for all $i\in\mathfrak{F}$.
            \end{remark}

            Theorem~\ref{theorem_nl} shows that the containment control problem of the general class of nonlinear systems \eqref{model_nl} can be achieved in the presence of unknown communication blackout intervals that can be arbitrarily large.  Note that the reference velocity $v_{r_i}$ in \eqref{reference_velocity_nonlinear} depends on the states of the nonlinear system \eqref{model_nl}, which introduces some coupling between the dynamics of this reference velocity and the tracking error $e_i = (v_i - v_{r_i})$. Theorem~\ref{theorem_nl} takes into account such coupling and provides sufficient conditions on the control input $\Gamma_i$, $i\in\mathfrak{F}$, such that tracking the reference velocity \eqref{reference_velocity_nonlinear}-\eqref{filter_observer:nonlinear} leads to our control objectives. Keeping in mind that $v_{r_i}$, $\dot{v}_{r_i}$ are well defined continuous functions of time and available for feedback, the design of such control input $\Gamma_i$ would be possible following the various approaches dealing with tracking control design for nonlinear systems. It is worthwhile mentioning that a similar approach has been recently considered in \cite{abdessameud2015synchronization} to address the synchronization problem of multi-agent system \eqref{model_nl} under similar assumptions on the communication constraints. The distributed design of the reference velocity  in Theorem~\ref{theorem_nl} extends our results in  \cite{abdessameud2015synchronization} to the case of multiple dynamic leaders with time-varying accelerations. Moreover, the algorithm proposed in Theorem~\ref{theorem_nl} is structurally simpler (as compared to the one proposed in \cite{abdessameud2015synchronization}) and, in addition, our control objectives can be attained without imposing any conditions on the control gains.

            To illustrate the application of Theorem~\ref{theorem_nl}, we consider the following dynamics of nonlinear systems
            \begin{equation}\label{nolinearity:example}
            F_{i}(p_i, v_i, \Gamma_i) := \bar{F}_{i}(p_i, v_i) + \Gamma_i,\qquad \mbox{for}~i\in\mathfrak{F},
            \end{equation}
            where $\bar{F}_i: \mathbb{R}^N\times \mathbb{R}^N \to \mathbb{R}^N$, is a locally Lipschitz continuous function satisfying the following assumption.
            \begin{assumption}\label{assm:nonlinearity}
            For all $i\in\mathfrak{F}$, there exists known class-$\mathcal{K}_{\infty}$ functions $\delta^p_{f_i}$ and $\delta^v_{f_i}$ such that
            \begin{equation}
               |\bar{F}_{i}(p_i, v_i)| \leq \delta^p_{f_i}(|p_i|) + \delta^v_{f_i}(|v_i|).
            \end{equation}
            \end{assumption}
            For each follower in \eqref{nolinearity:example}, consider the following control input
            \begin{eqnarray}\label{input:nl:example:vs}
            \Gamma_i &=& -k_{r_i} e_i + \dot{v}_{r_i} - \bar{\Gamma}_{i},\\
             \label{input:nl:2}\bar{\Gamma}_i &=& \left\{\begin{array}{ll}
            \frac{e_i}{|e_i|}(\delta^p_{f_i}(|p_i|) + \delta^v_{f_i}(|v_i|)), \quad& \mbox{if}~e_i \neq 0,\\
            0,          \quad&  \mbox{if}~e_i = 0\end{array}\right.
            \end{eqnarray}
            where $k_{r_i}>0$, $e_i = (v_i - v_{r_i})$, $v_{r_i}$ is given in \eqref{reference_velocity_nonlinear}-\eqref{filter_observer:nonlinear}. Note that the control algorithm \eqref{input:nl:example:vs} is a classical variable structure controller that satisfies the conditions in Theorem~\ref{theorem_nl} for a given $v_{r_i}$ and $\dot{v}_{r_i}$. In fact, using the Lyapunov function $V_i=0.5 e_i^\top e_i$ whose derivative evaluated along the closed loop dynamics satisfies $\dot{V}_i \leq -2k_{r_i} V_i$, one can show that $e_i$ is uniformly bounded and $e_i(t)\to 0$ for all $i\in\mathfrak{F}$. Therefore, the following corollary of Theorem~\ref{theorem_nl} holds.
             \begin{corollary}\label{corollary_nl}
                Consider multi-agent system \eqref{model_nl} with \eqref{nolinearity:example} and suppose that Assumptions~\ref{assumption_graph}-\ref{assm:nonlinearity} hold. For each follower, let the control input be given in \eqref{input:nl:example:vs}-\eqref{input:nl:2} with \eqref{reference_velocity_nonlinear}-\eqref{filter_observer:nonlinear}. Then,  $\bs{v}_F$, $\mbf{p}_F + (\mathcal{L}_1^{-1}\mathcal{L}_2\otimes I_N)\mbf{p}_L$ are uniformly bounded, $v_i(t)\to v_d$, for $i\in\mathfrak{F}$, and $\mbf{p}_F(t) + (\mathcal{L}_1^{-1}\mathcal{L}_2\otimes I_N)\mbf{p}_L(t)\to 0$, for arbitrary initial conditions.
                \hfill$\blacksquare$
            \end{corollary}

            \begin{remark}In \cite{wang:X:ISS:2014}, an ISS method has been proposed to solve the containment control problem for nonlinear multi-agent system \eqref{model_nl}, with \eqref{nolinearity:example} and Assumption~\ref{assm:nonlinearity}, where only the case of stationary leaders has been considered using continuous-time communication between agents and no communication delays. The result in \cite{wang:X:ISS:2014} is achieved under sufficient conditions given in terms of a number of inequalities that depends on the number of the directed paths and cycles in the directed interconnection graph between agents. Besides the fact that we consider communication constrains and non-stationary leaders, the results in Corollary~\ref{corollary_nl} do not rely on any centralized information on $\mathcal{G}$, and hold in the presence of unknown bounded blackout intervals that can be large.
            \end{remark}

           \begin{remark} The result of Theorem~\ref{theorem_nl} can be applied to various other nonlinear second order systems, including mechanical systems, which need not to be identical. \end{remark} 

\section{Containment Control of Oscillator Systems}

            Consider a multi-agent system where the dynamics of each agent are described by
        \begin{equation}\label{model_harmonic}
        \dot{p}_i = v_i, \qquad \dot{v}_i = S_1 p_i + S_2 v_i +  \Gamma_i, \quad i\in\mathcal{N},
        \end{equation}
            where  $\Gamma_i$ is the control input and $S_1\in\mathbb{R}^{N\times N}$, $S_2\in\mathbb{R}^{N\times N}$ are known matrices. Let $\Gamma_i\equiv 0$, for $i\in\mathfrak{L}$, and consider the following assumption.
        \begin{assumption}\label{assumption_harmonic} All eigenvalues of $S$ are pure imaginary and semi-simple, with $S:= \begin{bmatrix} 0_N& I_N\\ S_1& S_2\end{bmatrix}$.
        \end{assumption}

            Assumption~\ref{assumption_harmonic} implies that all agents oscillate with some frequency and some amplitude defined by their initial states if subject to no input. It is clear that $S_1 = S_2 = 0_N$ corresponds to the case of double integrators studied in Section~\ref{sec:linear:systems}. Also, if $S_1 = \mbox{diag}\{-s_{1,1}, \ldots, -s_{1,N}\}$, $s_{1,j}>0$ for $j=1,\ldots, N$, $S_2 = 0_N$, the dynamic system \eqref{model_harmonic} describes a group of harmonic oscillators studied in \cite{ren2008synchronization, su2009synchronization, zhou2012synchronization, xu2015containment}. In this section, we design the input of the follower agents such that their final trajectories are driven to the convex hull spanned by the trajectories of the leaders  in the presence of communication constraints.

            Let the control input in \eqref{model_harmonic} be as follows
        \begin{align}\label{input:harmonic}
        \Gamma_i &= - (k_{d_i}I_N+S_2) (v_i - \hat{v}_i) - k_{p_i}(p_i - \psi_{i,1})+ 2 \sigma_i \\
        \label{dot_sigma_input_harmonic}\dot{\sigma}_i&= S_1 (v_i - \hat{v}_i) + S_2 \sigma_i - k_{d_i}\sigma_i - k_{p_i}(\hat{v}_i - \psi_{i,2})\\
        \label{dot_hat_v_input_harmonic}\dot{\hat{v}}_i &= S_1 p_i + S_2 \hat{v}_i + \sigma_i
        \end{align}
            for $i\in\mathfrak{F}$, where $k_{p_i}>0$, $k_{d_i}>0$, and the vector $\psi_{i}:= (\psi_{i,1}^{\top}, \psi_{i,2}^{\top})^\top \in \mathbb{R}^{2N\times 2N}$ satisfies
        \begin{equation}\label{psi_input_harmonic}
        \psi_i 
        = \frac{1}{\kappa_i}\sum_{j=1}^n a_{ij} e^{S(t-t_{k_{ij}^{\mathrm{m}}(t)})}\begin{bmatrix}p_{j}(t_{k_{ij}^{\mathrm{m}}(t)}) \\ \hat{v}_j(t_{k_{ij}^{\mathrm{m}}(t)})\end{bmatrix}
            \end{equation}
            with $\hat{v}_i \equiv v_i$, for $i\in\mathfrak{L}$. The vector $\psi_i$ in \eqref{psi_input_harmonic} is defined based on an estimate, or a prediction, of the current positions of neighboring agents obtained using the most recent information available to agent $i$ at $t$. Also, similarly to the control algorithm \eqref{input:dynamic:1}-\eqref{psi_input_dynamic1}, the vectors $\psi_{i,1}$ and $\hat{v}_{i}$ can be interpreted as a reference position and reference velocity, respectively, for each follower designed to achieve our objectives. However, the dynamic system \eqref{dot_sigma_input_harmonic}-\eqref{dot_hat_v_input_harmonic}, $i\in\mathfrak{F}$, with $\hat{v}_i \equiv v_i$, $i\in\mathfrak{L}$, cannot be seen as an independent distributed observer since it relies on the positions of the agents in \eqref{dot_hat_v_input_harmonic} and \eqref{psi_input_harmonic}. In fact, the control algorithm \eqref{input:harmonic}-\eqref{psi_input_harmonic} does not reduce to \eqref{input:dynamic:1}-\eqref{psi_input_dynamic1} in the case of $S_1 = S_2 = 0_N$.

            In the case where the velocity vectors $v_i$ for $i \in\mathfrak{F}$ are not available for feedback, we implement the following control law for each follower
        \begin{align}
        \Gamma_i &= - (k_{d_i}I_N+S_2) \big(k_{\phi_i}(\phi_i + p_i) - \hat{v}_i\big) - k_{p_i}(p_i - \psi_{i,1})+ 2\sigma_i \nonumber\\
        \label{input:harmonic:withoutvelocity}\dot{\sigma}_i&= - (k_{d_i}I_N-S_2)\sigma_i + S_1 \big(k_{\phi_i}(\phi_i + p_i) - \hat{v}_i\big) - k_{p_i}(\hat{v}_i - \psi_{i,2})\nonumber\\
        \dot{\phi}_i &= -(k_{\phi_i}I_N - S_2)(p_i + \phi_i) + \frac{1}{k_{\phi_i}}(\Gamma_i + S_1 p_i)
        \end{align}
        with $k_{\phi_i}>0$ and $\hat{v}_i$, $\psi_i:= (\psi_{i,1}^{\top}, \psi_{i,2}^{\top})^\top $ being defined in \eqref{dot_hat_v_input_harmonic} and \eqref{psi_input_harmonic}, respectively.
        \begin{theorem}\label{theorem_harmonic}
                Consider the network of $n$ systems described by~\eqref{model_harmonic} and suppose Assumptions~\ref{assumption_graph}-\ref{assump_common_sampling} and Assumption~\ref{assumption_harmonic} hold. For each $i\in\mathfrak{F}$, consider the control input \eqref{input:harmonic}-\eqref{psi_input_harmonic}. Pick the gains $k_{p_i}$ and $k_{d_i}$ such that the roots of $x^2+k_{d_i} x+k_{p_i}=0$ are real. Then,
                all signals are
                uniformly bounded and $\bs{v}_F(t) + (\mathcal{L}_1^{-1}\mathcal{L}_2\otimes I_N)\bs{v}_L(t)\to 0$, $\mbf{p}_F(t) + (\mathcal{L}_1^{-1}\mathcal{L}_2\otimes I_N)\mbf{p}_L(t)\to 0$, for arbitrary initial conditions.
                The same results also hold in the case where the control algorithm \eqref{input:harmonic:withoutvelocity} with \eqref{dot_hat_v_input_harmonic}-\eqref{psi_input_harmonic} is used provided that the matrix $(-k_{\phi_i}I_N + S_2)$ is stable.
                \hfill$\blacksquare$
        \end{theorem}
\begin{proof}
                Let $x_i = \big(p_i^\top, \hat{v}_i^\top\big)^\top$, $i\in\mathcal{N}$,  $z_i =\big((v_i - \hat{v}_i)^\top, \sigma_i^\top\big)^\top$, $i\in\mathfrak{F}$, where $\hat{v}_i \equiv v_i$ for $i\in\mathfrak{L}$. Then, one can verify, from \eqref{model_harmonic} with \eqref{input:harmonic}-\eqref{psi_input_harmonic}, that $\dot{x}_i = S x_i$, $i\in\mathfrak{L}$, and
            \begin{equation}\label{closed_pf_thm_harmonic_1}\begin{array}{lcl}
            \dot{x}_i &=& S x_i + z_i \\
            \dot{z}_i &=& S z_i -k_{d_i}z_i - k_{p_i}(x_i - \psi_i)+ \bs{\epsilon}_i
            \end{array}, \quad i\in\mathfrak{F}\end{equation}
                where $\psi_i$ is defined in \eqref{psi_input_harmonic} and $\bs{\epsilon}_i \equiv 0$. Similarly to the proof of Theorem~\ref{theorem_nonstationary}, we suppose that $\bs{\epsilon}_i$ is uniformly bounded and $\bs{\epsilon}_i(t)\to 0$ for analysis purposes. In view of Assumption~\ref{assumption_harmonic}, we consider the following change of variables $\bar{x}_i(t) = e^{-S(t-t_0)} x_i(t)$, for $i\in\mathcal{N}$ and $t > t_0\geq 0$, and define similarly $\bar{z}_i(t)$, $\bar{\psi}_i(t)$, $\bar{\bs{\epsilon}}_i(t)$, for $i\in\mathfrak{F}$. This leads to
            \begin{equation}\label{dot_z:1:_pf_thm_harmonic}\begin{array}{lcl}
            \dot{\bar{x}}_i &=& \bar{z}_i\\
            \dot{\bar{z}}_i &=& -k_{d_i} \bar{z}_i - k_{p_i}\big(\bar{x}_i - \bar{\psi}_i \big)+ \bar{\bs{\epsilon}}_i
            \end{array}, \quad i\in\mathfrak{F}
            \end{equation}
                with \vspace{-0.1 in}
            \begin{equation*}  \begin{array}{lcl}
            \bar{\psi}_i &=&
                        \frac{1}{\kappa_i}\sum_{j=1}^n a_{ij} e^{-S(t-t_0)}e^{S(t-t_{k_{ij}^{\mathrm{m}}(t)})}x_{j}(t_{k_{ij}^{\mathrm{m}}(t)}) \\
            &=&   \frac{1}{\kappa_i}\sum_{j=1}^n a_{ij}\bar{x}_j(t_{k_{ij}^{\mathrm{m}}(t)}).
          \end{array}\end{equation*}

                Now, to use the result of Lemma~\ref{theorem_unified}, we consider the change of variables $\bar{\xi}_i = \frac{1}{\lambda_i}\bar{z}_i + \bar{x}_i$, where $-\lambda_i <0$ is a real root of the characteristic equation given in the theorem. Then, one can show, using similar steps as in \eqref{dot_xi_pf_thm1}, that
            \begin{equation}\label{closed_pf_thm_harmonic}\begin{array}{lcl}
            \dot{\bar{x}}_i = 0&\qquad& i\in\mathfrak{L}\\
            \dot{\bar{x}}_i = -\lambda_i (\bar{x}_i - \bar{\xi}_i)&\qquad& i\in\mathfrak{F}\\
            \end{array}\end{equation} with
            \begin{equation}\label{dot_xi_pf_thm_harmonic}
            \dot{\bar{\xi}}_i = \frac{-k_{p_i}}{\lambda_i}\big(\bar{\xi}_i - \frac{1}{\kappa_i}\sum_{j=1}^n a_{ij}\bar{x}_j(t_{k_{ij}^{\mathrm{m}}(t)}) \big) + \frac{1}{\lambda_i}\bar{\bs{\epsilon}}_i
            \end{equation}
                for $i\in\mathfrak{F}$. Then, since the dynamic system  \eqref{closed_pf_thm_harmonic}-\eqref{dot_xi_pf_thm_harmonic} can be rewritten as \eqref{filter_leader_lemma}-\eqref{gains_filter} with $\alpha = 1$ and $\sigma_i = 1$, $i\in\mathfrak{F}$,  Lemma~\ref{theorem_unified} can be used to show that $\dot{\bar{\bs{\xi}}}_F$, $\dot{\bar{\bs{x}}}_F$, $\bar{\bs{x}}_F + (\mathcal{L}_1^{-1}\mathcal{L}_2\otimes I_N)\bar{\bs{x}}_L$ are uniformly bounded and asymptotically converge to zero under the assumptions of the theorem. This implies that $\dot{\bar{\bs{z}}}_F$, $\bar{\bs{z}}_F$ are uniformly bounded and $\dot{\bar{\bs{z}}}_F(t)\to 0$, $\bar{\bs{z}}_F(t)\to 0$, in view of \eqref{dot_z:1:_pf_thm_harmonic}. This, with Assumption~\ref{assumption_harmonic}, leads to the conclusion that $\dot{\bs{z}}_F$, $\bs{z}_F$, $\dot{\bs{x}}_F$, $\bs{x}_F + (\mathcal{L}_1^{-1}\mathcal{L}_2\otimes I_N)\bs{x}_L$ are uniformly bounded and asymptotically converge to zero, which leads to the result of the theorem from the definition of $\bs{z}_F$ and $\bs{x}_F$. Also, $\bs{x}_F$ is uniformly bounded since $\bs{x}_L$ is uniformly bounded by Assumption~\ref{assumption_harmonic}.

                In the case where the control law \eqref{input:harmonic:withoutvelocity} is used, we define the error vector
                $\tilde{\phi}_i := \phi_i + p_i - \frac{1}{k_{\phi_i}}v_i$. Using \eqref{model_harmonic} with \eqref{input:harmonic:withoutvelocity}, one can show that $$\dot{\tilde{\phi}}_i = - (k_{\phi_i} I_N - S_2)\tilde{\phi}_i,$$
                for $i\in\mathfrak{F}$. This implies that $\tilde{\phi}_i(t)\to 0$ exponentially if $- (k_{\phi_i} I_N - S_2)$ is stable. Also, by noticing that $v_i = k_{\phi_i}(\phi_i + p_i -\tilde{\phi}_i)$, and using the same variables $x_i$, $z_i$, defined above, the closed loop dynamics \eqref{model_harmonic} with \eqref{input:harmonic:withoutvelocity}, \eqref{dot_hat_v_input_harmonic}-\eqref{psi_input_harmonic}, can be written as in \eqref{closed_pf_thm_harmonic}-\eqref{dot_xi_pf_thm_harmonic} with $\bs{\epsilon}_i = (\bs{\epsilon}_{i,1}^\top, \bs{\epsilon}_{i,2}^\top)^\top$, ~ $\bs{\epsilon}_{i,1}:=-k_{\phi_i}(k_{d_i}I_N +S_2) \tilde{\phi}_i$, and $\bs{\epsilon}_{i,2}:= k_{\phi_i}S_1\tilde{\phi}_i$. The same above arguments can then be used to prove the last part of the theorem.
%
\end{proof}
         
             Theorem~\ref{theorem_harmonic} provides a solution to the containment control problem of linear oscillators \eqref{model_harmonic} under relaxed assumptions on the communication process between agents (leading to  unknown communication blackout intervals that can be arbitrarily large) without imposing additional restrictions on the directed graph. This is guaranteed with a simple choice of the control gains despite the oscillatory motion of all agents. Further, the above result removes the requirements of velocity measurements for the followers.              It is clear that the distributed control algorithm in Theorem~\ref{theorem_harmonic} can be applied, with an obvious modification, to the case where the leaders' trajectories are oscillating, according to  \eqref{model_harmonic} with  Assumption~\ref{assumption_harmonic}, and the follower agents are governed by double integrator dynamics  \eqref{model} and/or the nonlinear second-order dynamics \eqref{model_nl}.

    \begin{remark}
        The synchronization problem of harmonic oscillators has been addressed in the literature in both the leaderless and leader-follower scenarios in the case of continuous-time communication \cite{ren2008synchronization, su2009synchronization} and in the case where communication is lost during some intervals of time \cite{zhou2012synchronization}. More recently, the containment control problem of coupled harmonic oscillators in directed networks has been studied in \cite{xu2015containment} using sampled-data protocols. Communication constraints, however, have not been considered in these papers and only the case of undirected interconnection graphs is addressed in \cite{zhou2012synchronization}.
    \end{remark}

\section{Simulation Results }\label{sec:simulation}

            In this section, we implement the proposed distributed containment control schemes for a network of ten systems, with  $\mathfrak{F} =\{1, \ldots, 6\}$ and $\mathfrak{L}= \{7, \dots, 10\}$, moving in the two dimensional space with $N=2$.
         The communication process between agents is described in Section~\ref{sec:comm} with the parameter $T^*$ in Assumption~\ref{AssumptionCommunicationBlackouts} being estimated to be smaller than or equal to $1.5~\sec$, and the interconnection graph $\mathcal{G}$ is directed and satisfies Assumption~\ref{assumption_graph} with
         \[
         \setlength{\arraycolsep}{2pt}
  \renewcommand{\arraystretch}{1}
  \mathcal{L}_1 = \begin{bmatrix}
    2 &0& -1& 0& 0& 0\\
            0 &2& -1& 0& 0& 0\\
            -1& 0& 4& -1& 0& 0\\
            0& 0& -1& 3& -1& -1\\
            0& 0& 0& 0& 3& -1\\
            0& 0& 0& 0& -1& 2\\
  \end{bmatrix}, \quad
         \mathcal{L}_2 = \begin{bmatrix}
            -1& 0& 0& 0\\
            0& -1& 0& 0\\
            0& -1& -1& 0\\
            0& 0& 0& 0\\
            0& -1& -1& 0\\
            0& 0& 0& -1\\
            \end{bmatrix}.
\]
        \noindent Also, we let $p_i = (p_{i_1}, p_{i_2})^\top\in\mathbb{R}^2$, $v_i = (v_{i_1}, v_{i_2})^\top\in\mathbb{R}^2$, and the vector $\tilde{\bs{v}}_{F}$ denote the containment error of the system defined as $\tilde{\bs{v}}_{F}=(\tilde{\bs{v}}_{F_1}^{\top}, \ldots, \tilde{\bs{v}}_{F_6}^\top)^\top := \bs{v}_F + (\mathcal{L}_1^{-1}\mathcal{L}_2\otimes \mbf{I}_2)\bs{v}_L$, where $\tilde{\bs{v}}_{F_i} = (\tilde{\bs{v}}_{F_{i_1}}, \tilde{\bs{v}}_{F_{i_2}})^{\top}\in\mathbb{R}^2$, $i\in\mathfrak{F}$.
\vspace{0.05 in}\\
          \emph{Example 1}:~ We consider the follower agents modeled by \eqref{model} and the motion of the leaders is described by $v_i(t) = v_i(0) + \mbf{1}_2(-1)^i\left(\cos(t)-0.2\sin(t)\right)e^{-0.2t}i$, for $i\in\mathfrak{L}$, with $v_i(0)$ defined such that $v_i(t)\to (1,0.1)^\top$, for $i\in\mathfrak{L}$, such that Assumption~\ref{asum:leader:motion} is verified. We implement the control algorithm \eqref{input:dynamic:1}-\eqref{psi_input_dynamic1} in Theorem~\ref{theorem_nonstationary}  with the control gains $k_{p_i}=k_{d_i}=L_{p_i}=4$, $i\in\mathfrak{F}$. In the case where the velocity vectors $v_i$, $i\in\mathfrak{F}$, are not available for feedback, we implement the control algorithm \eqref{input:dynamic:withoutvelocity}-\eqref{dot_phi_input_withoutvelocity} with \eqref{filter_observer_dynamic} and \eqref{dot_psi_input_dynamic} with the same above gains and $k_{\psi_i}= 1$. It can be verified that, in both cases, $k_{d_i}=2\sqrt{k_{p_i}}$ and hence the condition in Theorem~\ref{theorem_nonstationary} is satisfied.

          Fig.~\ref{fig:containment:error:variable:nonstationary} shows the convergence of the containment error vectors for all followers to zero in the case where control algorithm \eqref{input:dynamic:1}-\eqref{psi_input_dynamic1} of Theorem~\ref{theorem_nonstationary}, which implies that all followers converge to the convex hull spanned by the leaders using intermittent discrete-time communication, in the presence of time-varying delays and packets dropouts. This is also demonstrated in Fig.~\ref{fig:positions:2D:nonstationary}, which depicts the trajectories of all agents at different instants of time in this case. The same conclusions can be drawn from  Fig.~\ref{fig:positions:2D:partial} showing the results obtained in the partial state feedback case. Note that the trajectories of the leaders are the same in both cases and the trajectories of the followers are shown only in Fig.~\ref{fig:positions:2D:partial} for clarity.

\begin{figure}[h]\vspace{-0.1 in}
\centering{\includegraphics[width = 0.90 \columnwidth]{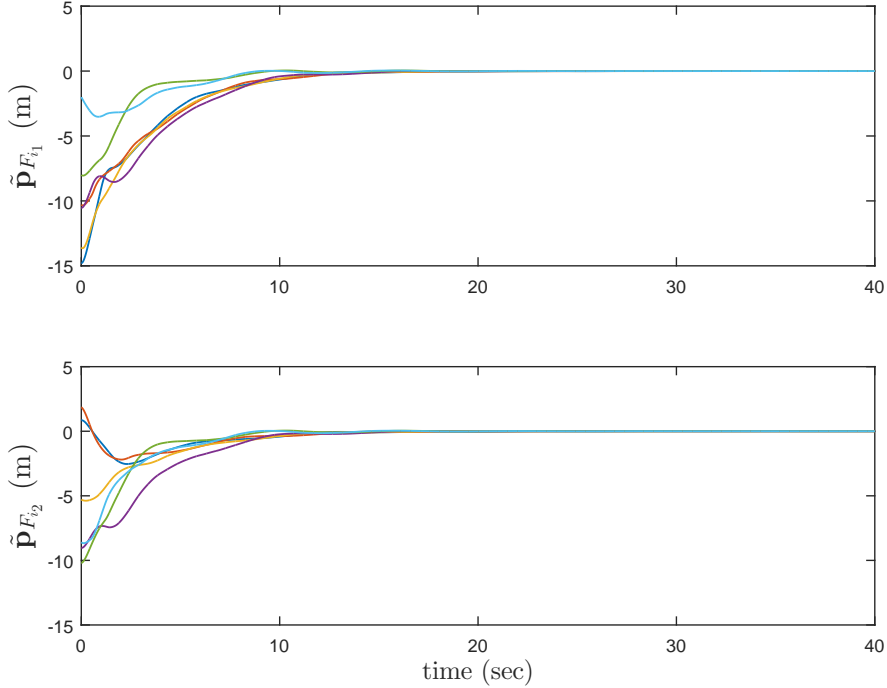}}
\vspace{-0.1 in}
\caption{{\footnotesize{Containment errors in Example 1 with control algorithm \eqref{input:dynamic:1}-\eqref{psi_input_dynamic1}.}}}
\label{fig:containment:error:variable:nonstationary}
\end{figure}

\begin{figure}[h]\vspace{-0.1 in}
\centering
{\includegraphics[width = 0.90\columnwidth]{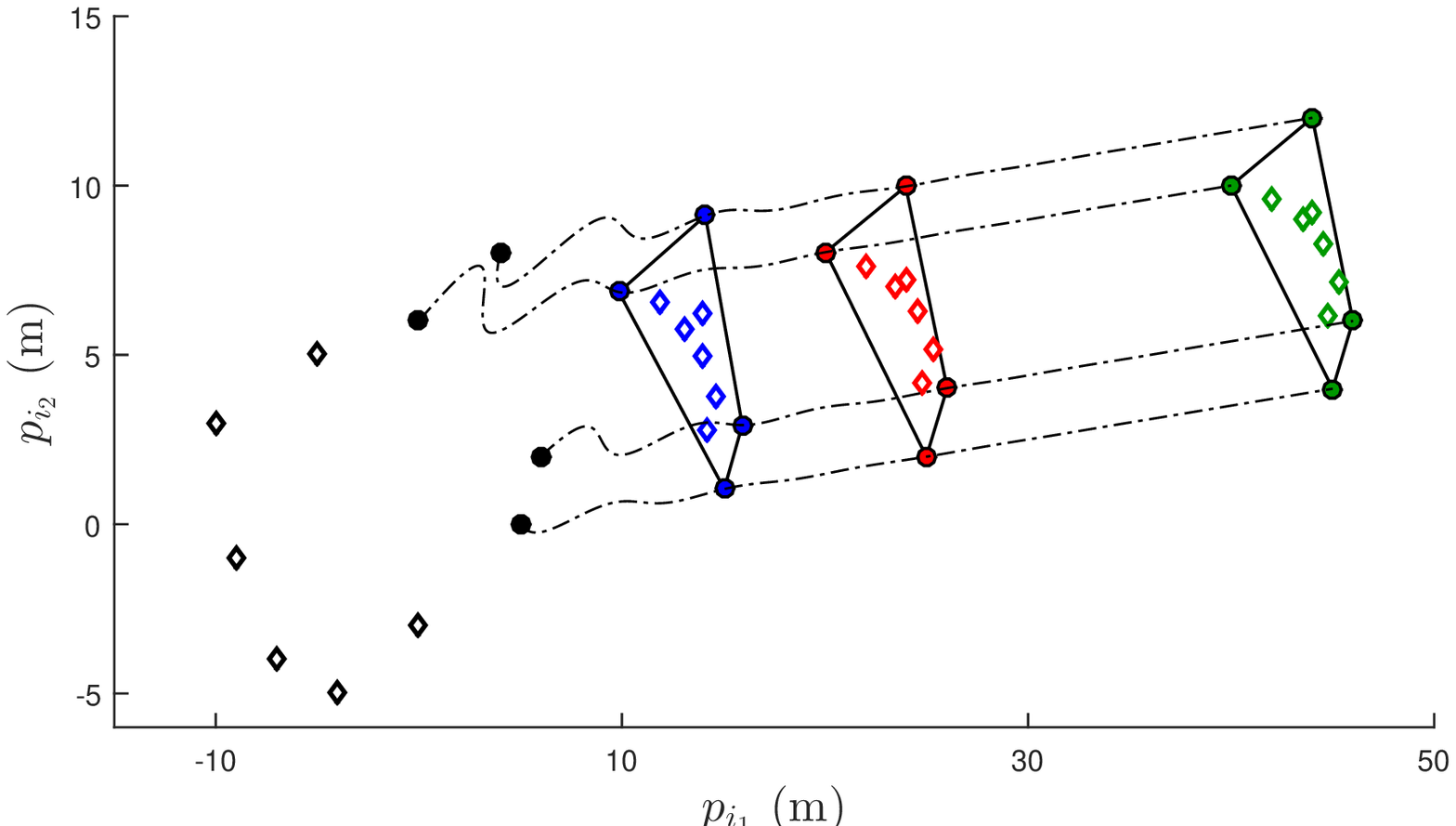}}
\vspace{-0.05 in}
\caption{{\footnotesize{Simulation results for Example 1 with control algorithm \eqref{input:dynamic:1}-\eqref{psi_input_dynamic1}: Trajectories of the leaders (black dashed lines), convex hull of the leaders at different instants of time (closed shape with black sides), positions of the six followers (diamonds) and four leaders (circles) at instants: 0 $\sec$ (black), 10 $\sec$ (blue), 20 $\sec$ (red), and 40 $\sec$ (green). }}}
\label{fig:positions:2D:nonstationary}
\end{figure}

\begin{figure}\vspace{-0.1 in}
\centering{
\includegraphics[width = 0.90\columnwidth]{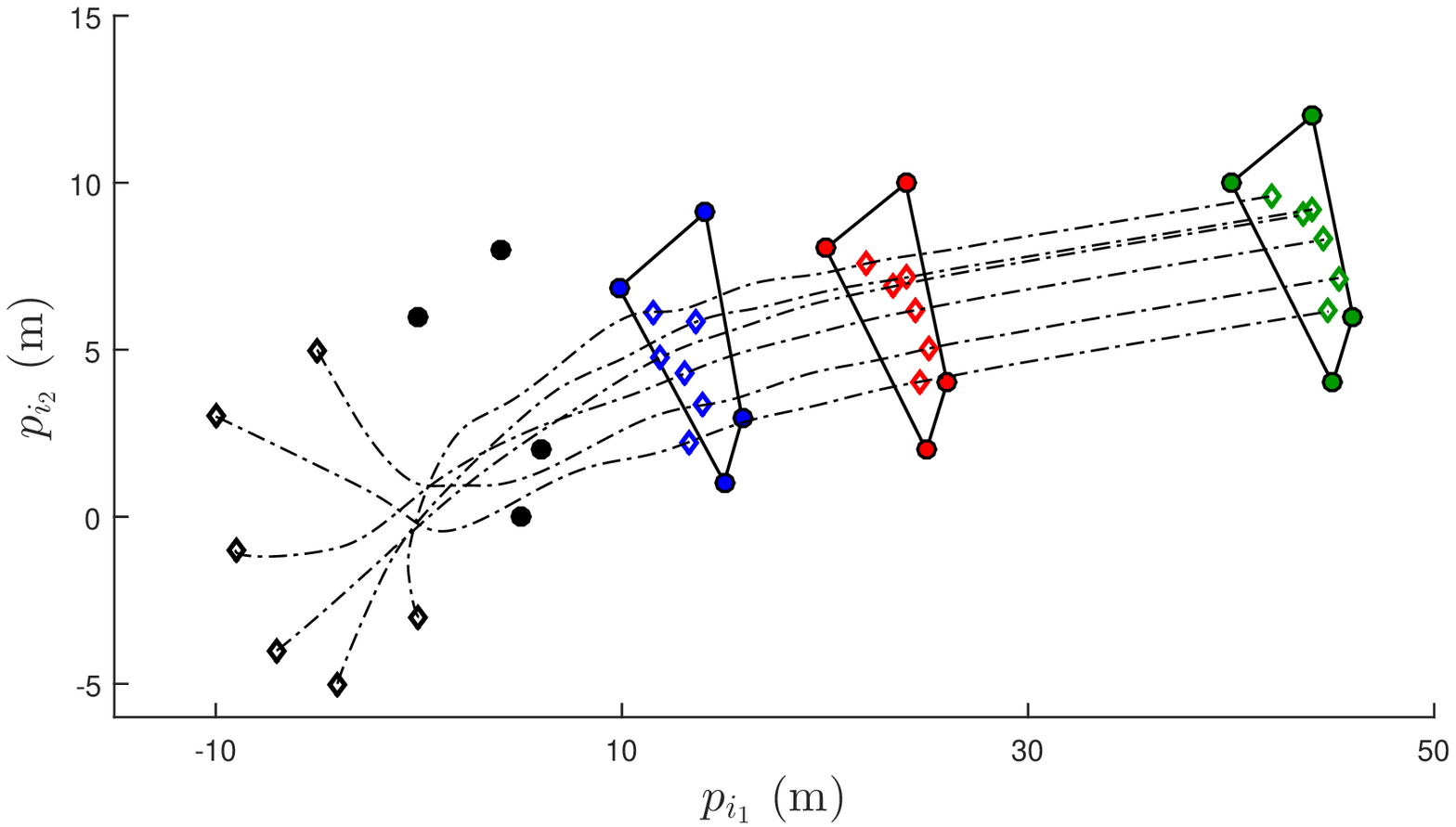}}
\vspace{-0.05 in}
\caption{{\footnotesize{Simulation results for Example 1 with control algorithm \eqref{input:dynamic:withoutvelocity}-\eqref{dot_phi_input_withoutvelocity}, \eqref{filter_observer_dynamic}, and \eqref{dot_psi_input_dynamic}: Trajectories of the followers (black dashed lines), convex hull of the leaders at different instants of time (closed shape with black sides), positions of the six followers (diamond) and the four leaders (circles) at instants: 0 $\sec$ (black), 10 $\sec$ (blue), 20 $\sec$ (red), and 40 $\sec$ (green). }}}
\label{fig:positions:2D:partial}
\end{figure}

        \noindent\emph{Example~2}:~ We consider the case where the followers are governed by the nonlinear dynamics \eqref{model_nl} with \eqref{nolinearity:example}. In particular, and similarly to \cite{wang:X:ISS:2014}, we let $$F_i(p_i, v_i, \Gamma_i) = \Gamma_i + (v_{i_1}^2, v_{i_{2}}^2)^{\top}, \quad~i\in\mathfrak{F},$$ which satisfies Assumption~\ref{assm:nonlinearity} with $\delta^p_{f_i}(|p_i|)=0$ and $\delta^q_{f_i}(|v_i|)=|v_i|^2$. Then, for all follower agents, we implement the control algorithm described in Corollary~\ref{corollary_nl} with the control gains $k_{r_i}=1$, $\lambda_i = 2$, $k_{p_i}=k_{d_i}=L_{p_i}=L_{d_i} = 2$, $i\in\mathfrak{F}$.

        To validate the result in Proposition~\ref{theorem_formation_leaders}, we assume that the leaders can transmit their information, using the same communication process, under the directed interconnection graph $\mathcal{G}_{\mathfrak{L}}$ having the set of edges $\mathcal{E}_{\mathfrak{L}}=\{(7,8),(8,9),(10,7)\}$, which is a spanning tree rooted at node $10$. Then, we implement the algorithm in Proposition~\ref{theorem_formation_leaders} with arbitrary initial conditions of the leader agents and  $\bar{v}_d(t) = (1,0.1)^\top + (\cos(t)-0.2\sin(t))e^{-0.2t}~\mathrm{m/\sec}$. Also, the control gains are selected as: $k_{d_{i}}=4$, $i\in\mathfrak{L}$, $k_{p_i}=4$, $k_{\psi_i}=1$, $i\in\{7,8,9\}$.  The desired separation vectors between the leader agents are considered as $\delta_{ij}=\delta_i - \delta_j$, with $\delta_7=(1,-2)^\top$, $\delta_{8}=(2,1)^\top$, $\delta_{9}=(-1,2)^\top$ and $\delta_{10}=(-2,-1)^\top$. The obtained results in this case are given in Fig.~\ref{fig:positions:2D:nl:form}, which shows that all the followers converge to the convex hull spanned by the leaders and the leaders converge to the specified geometric shape with the desired velocity, using intermittent and delayed communication between agents.

\begin{figure}[h]\vspace{-0.1 in}
\centering
{
\includegraphics[width = 0.90\columnwidth]{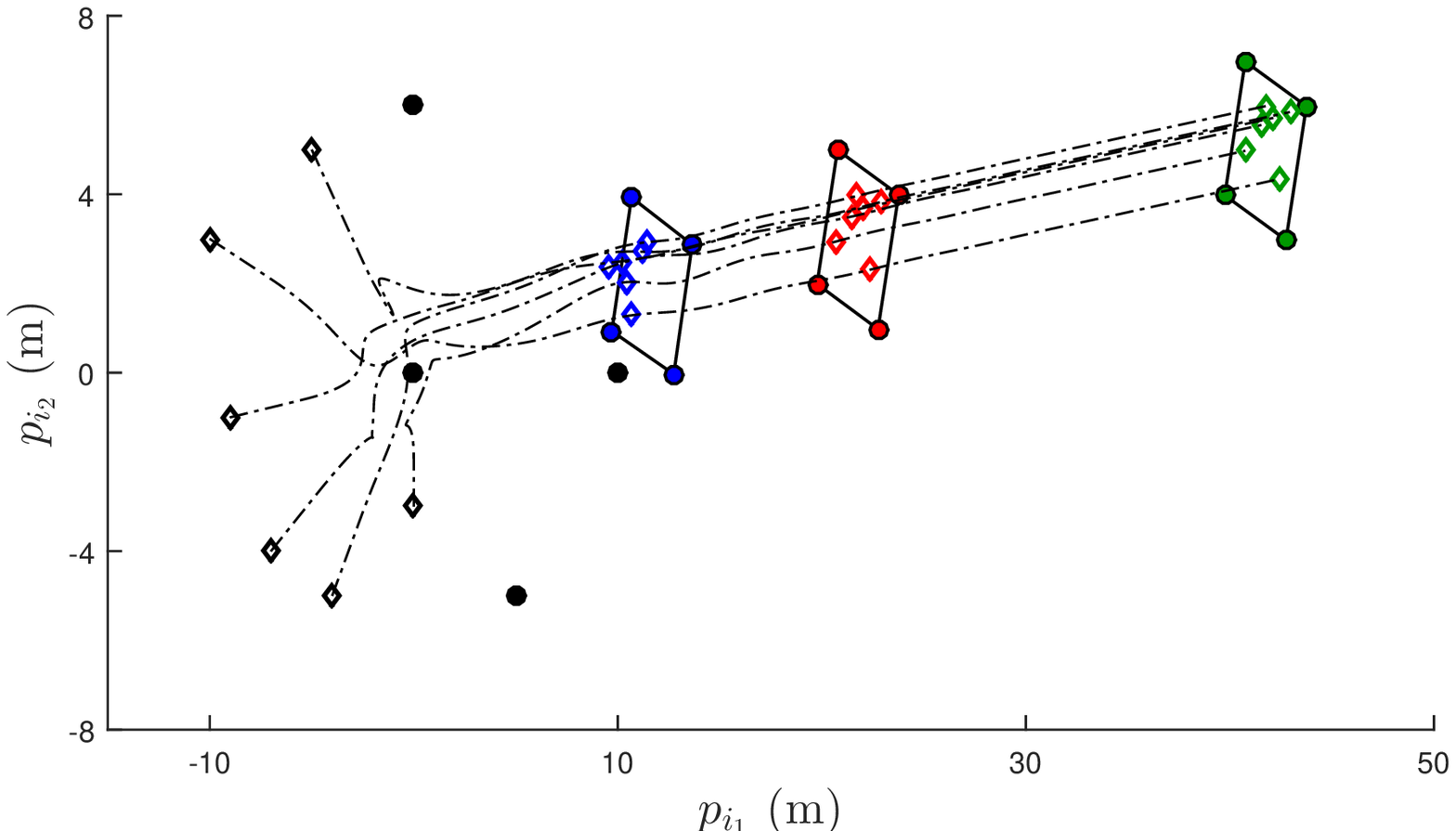}}
\vspace{-0.05 in}
\caption{{\footnotesize{Simulation results for Example 2: Trajectories of the followers (black dashed lines), convex hull of the leaders at different instants of time (closed shape with black sides), positions of the six followers (diamonds) and the four leaders (circles) at instants: 0 $\sec$ (black), 10 $\sec$ (blue), 20 $\sec$ (red), and 40 $\sec$ (green). }}}
\label{fig:positions:2D:nl:form}
\end{figure}

\section{Conclusion}
        In this paper, we presented solutions to the containment control problem with multiple dynamic leaders for linear and  nonlinear second-order multi-agent systems under mild assumptions on the communication and interconnection between agents. In all the proposed control algorithms, the information exchange is assumed to be intermittent with irregular time-varying communication delays and possible information losses. Our results are guaranteed under simple design conditions on the control parameters that can be realized independently from the characteristics of the communication process and without an {\it a priori} knowledge on the general directed interconnection graph topology. These distinctive features make our results fundamentally different from the available relevant literature as discussed throughout the paper.


\appendix

\section{proof of Lemma~\ref{theorem_unified}}\label{section_proofs}

         Before addressing the proof of Lemma~\ref{theorem_unified}, we recall some definitions on input-to-state and input-to-output stability notions as well as a small gain theorem proved in our earlier work \cite{Abdessameud:Polushin:Tayebi:2013:ieeetac}.
%
\noindent {\it Definitions:}~ Consider an affine nonlinear system of the form
            \begin{equation}
            \label{affine001}
            \begin{array}{rcl}
            {\dot x}&=&f(x)+ \displaystyle\sum_{i=1}^p g_i(x)u_i\\
            y_j & = & h_j(x),\quad j=1,\ldots, q,
            \end{array}
            \end{equation}
        where  ${x}\in\mathbb{R}^N$, ${u}_i\in\mathbb{R}^{\tilde{m}_i}$ for $i\in\mathcal{N}_p:=\{1,\ldots,p\}$, ${y}_j\in\mathbb{R}^{\bar{m}_j}$ for $j\in\mathcal{N}_q:=\{1,\ldots,q\}$, and $f(\cdot )$, $g_i(\cdot )$, for $i\in\mathcal{N}_p$, and $h_j(\cdot )$, for $j\in\mathcal{N}_q$, are locally Lipschitz functions of the corresponding dimensions, $f(0)=0$, $h(0)=0$. We assume that for any initial condition $x(t_0)$ and any inputs $u_1(t)$, \ldots, $u_p(t)$ that are uniformly essentially bounded on $[t_0, t_1)$, the corresponding solution $x(t)$ is well defined  for all $t\in [t_0, t_1]$.
        \begin{definition}\cite{sontag:06:1}
        \label{def_iss}
           A system of the form \eqref{affine001} is said to be input-to-state stable (ISS) if there exist\footnote{ The definition of the class functions $\mathcal{K}$, $\mathcal{K}_{\infty}$, and $\mathcal{KL}$ can be found in ~\cite{khalil:02}. Also, ${\bar {\mathcal K}}:={\mathcal K}\cup\{{\mathcal O}\}$, ${\bar {\mathcal K}}_{\infty}:={\mathcal K}_{\infty}\cup\{{\mathcal O}\}$, where ${\mathcal O}$ is zero function, ${\mathcal O}(s)\equiv 0$ for all $s\ge 0$.}
                 $\beta\in {\mathcal K}_{\infty}$ and $\gamma_{j}\in {\bar{\mathcal K}}$, $j\in\mathcal{N}_p$, such that the following inequalities hold along the trajectories of the system for any Lebesgue measurable uniformly essentially bounded inputs $u_j$, $j\in\mathcal{N}_p$:
            	\begin{itemize}
              	\item [i)] 
              $\forall~t_0,~ t\in{\mathbb R},~ t\ge t_0$, we have \vspace{-0.2 cm}
              	$$\left|x(t)\right|\le \beta\left(|x(t_0)|\right)+\sum\limits_{j=1}^{p} \gamma_{j}\left(\sup\limits_{s\in\left[t_0, t\right)}|u_j(s)|\right),$$
              	\item[ii)] 
              	$\limsup\limits_{t\to +\infty} \left|x(t)\right|\le \sum\limits_{j=1}^{p}\gamma_{j}\left(\limsup\limits_{t\to +\infty} |u_j(t)|\right).$
            	\end{itemize}
           \end{definition}

        In the above definition, $\gamma_j\in\bar{\mathcal{K}}$, $j\in\mathcal{N}_p$, are called the ISS gains. It should be pointed out that for a system of the form \eqref{affine001}, the ISS implies the input-to-output stability  (IOS) \cite{sontag:06:1}, which means that there exist $\beta_i\in \mathcal{KL}$ and $\gamma_{ij}\in {\bar{\mathcal K}}$, $i\in\mathcal{N}_q$, $j\in\mathcal{N}_p$, such that the inequality
        $$\left|y_i(t)\right|\le \beta_i\left(|x(t_0)|, t \right)+ \sum\limits_{j=1}^{p} \gamma_{ij}\left(\sup\limits_{s\in\left[t_0, t\right)}|u_j(s)|\right)$$
        holds for all $i\in\mathcal{N}_q$  and $\forall~t_0, ~t\in{\mathbb R},~ t\ge t_0$.
        In this case, the function $\gamma_{ij}\in {\bar {\mathcal K}}$, $i\in\mathcal{N}_q$ and $j\in\mathcal{N}_p$, is called the IOS gain from the input $u_j$ to the output $y_i$. In this paper, we mostly deal with the case where the IOS gains are linear functions of the form $\gamma_{ij}(s):=\gamma_{ij}^0\cdot s$,  where $\gamma_{ij}^0\ge 0$; in this case, we simply say that the system has linear IOS gains $\gamma_{ij}^0\ge 0$.

 Consider the following IOS small-gain theorem, which can be proved following similar steps as in the proof of \cite[Theorem 1]{Abdessameud:Polushin:Tayebi:2013:ieeetac}.

\begin{theorem} \label{theorem001a}
          Consider a system of the form \eqref{affine001}. Suppose the system is IOS with linear IOS gains $\gamma_{ij}^0\ge 0$. Suppose also that each input $u_j(\cdot)$, $j\in\mathcal{N}_p$, is a Lebesgue measurable function satisfying:
                     $u_j(t)\equiv 0$, for $t<0$, 
          and                     \begin{equation}
                            |u_j(t)|\le\sum\limits_{i\in\mathcal{N}_q} {\mu}_{ji}\cdot\sup\limits_{s\in\left[t-\vartheta_{ji}(t), t\right]} \left| y_i(s)\right| +|\delta_j(t)|,
                            \label{commconstraints0010}
                    \end{equation}
          for almost all $t\ge 0$, where $ {\mu}_{ji}\ge 0$,  all  $\vartheta_{ji}(t)$ are Lebesgue measurable uniformly bounded nonnegative functions of time, and $\delta_j(t)$ is an uniformly essentially bounded signal (uniformly bounded almost everywhere except for a set of measure zero). Let $\mathscr{G}:=\mathscr{G}^0\cdot {\mathcal M}\in{\mathbb R}^{q\times q}$, where $\mathscr{G}^0:=\left\{ \gamma^0_{ij} \right\}$, ${\mathcal M}:=  \left\{{\mu}_{ji}\right\}$, $i\in\mathcal{N}_q$, $j\in\mathcal{N}_p$. If $\bs{\rho}(\mathscr{G})<1$,
          then the trajectories of the system \eqref{affine001} with input-output constraints \eqref{commconstraints0010} are well defined for all $t\ge 0$ and such that all the outputs  $y_i(t)$, $i\in\mathcal{N}_q$, and all the inputs  $u_j(\cdot)$, $j\in\mathcal{N}_p$,  are uniformly bounded. If, in addition, $|\delta_j(t)|\to 0$, $j\in\mathcal{N}_p$, then $\left|y_i(t)\right|\to 0$, $\left|u_j(t)\right|\to 0$ for $i\in\mathcal{N}_q$ and $j\in\mathcal{N}_p$.
           	$\square$
    \end{theorem}

           Now, we are ready to proof Lemma~\ref{theorem_unified}. First, note that Assumption~\ref{assumption_graph} ensures that $\kappa_i$ in \eqref{input_filter_lemma} satisfies $\kappa_i\neq 0$ for $i\in\mathfrak{F}$.
             Also, Assumption~\ref{assumption_graph} and Lemma~\ref{lemma_Mei} ensure that $\mathcal{L}_1$ is a non-singular M-matrix. Let $$\bs{\eta}_{c} = \big(\eta_{c_1}^{\top}, \ldots, \eta_{c_m}^{\top}\big)^{\top} := (-\mathcal{L}_1^{-1}\mathcal{L}_2\otimes I_N) \bs{\eta}_L \in\mathbb{R}^{mN},$$ with $\eta_{c_i}\in\mathbb{R}^N$ for $i\in\mathfrak{F}$ and $\mathcal{L}_1$, $\mathcal{L}_2$ are defined in \eqref{laplacian}.

             Consider the error vectors
             \begin{equation}\label{tilde_eta_delta}
             \tilde{\eta}_i = \eta_i - \eta_{c_i}, \quad \tilde{\delta}_i = \delta_i - \eta_{c_i}, \quad \tilde{\zeta}_i = \zeta_i -\bar{\eta}_{c_i}
             \end{equation}
             for $i\in\mathfrak{F}$, with $\bar{\eta}_{c_i}:=(\mbf{1}_{\sigma_i}\otimes I_N) \eta_{c_i}\in\mathbb{R}^{\sigma_i N}$. Using \eqref{filter_follower_lemma}-\eqref{filter:lemma}, we can verify that
             \begin{eqnarray}
             \label{closed:dynamics_pflemma1:1}\dot{\tilde{\eta}}_i&=& -k_{\eta_i}\tilde{\eta}_i + k_{\eta_i}\tilde{\delta}_i + \Phi_{i,1} - \dot{\eta}_{c_i}\\
             \dot{\tilde{\zeta}}_i &=& \bar{H}_i \tilde{\zeta}_i + \bar{H}_i \bar{\eta}_{c_i} + \bar{B}_i \varepsilon_i + \Phi_{i,2} - \dot{\bar{\eta}}_{c_i}\\
             \label{closed:dynamics_pflemma1:3}\tilde{\delta}_i &=& \alpha \bar{C}_i \tilde{\zeta}_i + \alpha \bar{C}_i \bar{\eta}_{c_i} + (1-\alpha)\varepsilon_i - \eta_{c_i}
             \end{eqnarray}
            where we used notation $\bar{H}_i=(H_i\otimes I_N)$, $\bar{B}_i = (B_i \otimes I_N)$, $\bar{C}_i = (C_i\otimes I_N)$ for simplicity. From \eqref{gains_filter}, it is straightforward to show that $\bar{H}_i \bar{\eta}_{c_i} + \bar{B}_i \varepsilon_i = \bar{B}_i (\varepsilon_i - \eta_{c_i})$ and $\bar{C}_i \bar{\eta}_{c_i} = \eta_{c_i}$, $i\in\mathfrak{F}$.  Also, from the definition of $\bs{\eta}_{c}$ and \eqref{D}-\eqref{laplacian}, one can verify that
            \begin{align}
            (\mathcal{D}_1\otimes I_N) \bs{\eta}_c 
            &= -(\mathcal{L}_2\otimes I_N)\bs{\eta}_L +  (\mathcal{A}_1\otimes I_N)\bs{\eta}_c
            \end{align}
            where we used relation $\mathcal{D}_1 = \mathcal{L}_1 + \mathcal{A}_1$. Since the $(i,i)$-th entry of $\mathcal{D}_1$ is $\kappa_i$, we obtain
            \begin{equation}\label{eta_ci_pf_lemma}
            \eta_{c_i} = \frac{1}{\kappa_i}\Big(\sum_{j=1}^m a_{ij} \eta_{c_j} + \sum_{j=m+1}^n a_{ij} \eta_{j}\Big), \quad i\in\mathfrak{F}.
            \end{equation}
            Also, the input vector $\varepsilon_i$, in view of \eqref{input_filter_lemma} and \eqref{tilde_eta_delta}, satisfies
            \begin{equation*}
             \varepsilon_i = \frac{1}{\kappa_i}\Big(\sum_{j=1}^m a_{ij}\big(\tilde{\eta}_j(t_{k_{ij}^{\mathrm{m}}(t)}) + \eta_{c_j}(t_{k_{ij}^{\mathrm{m}}(t)})\big) +\sum_{j=m+1}^n a_{ij}\eta_j(t_{k_{ij}^{\mathrm{m}}(t)})\Big)
            \end{equation*}
            for $i\in\mathfrak{F}$. Then, it can be deduced from the last two equations that
            $\varepsilon_i - \eta_{c_i} = u_i$, $i\in\mathfrak{F}$, with
            \begin{align}
            \label{input_for_ISS}
            u_i =&~ \frac{1}{\kappa_i}\sum_{j=1}^m a_{ij} \tilde{\eta}_j (t_{k_{ij}^{\mathrm{m}}(t)}) - \frac{1}{\kappa_i} \sum_{j=1}^m a_{ij}\big(\eta_{c_j} - \eta_{c_j}(t_{k_{ij}^{\mathrm{m}}(t)})\big)\nonumber\\
            &~ - \frac{1}{\kappa_i}\sum_{j=m+1}^n a_{ij}\big(\eta_{j} - \eta_{j}(t_{k_{ij}^{\mathrm{m}}(t)})\big).
            \end{align}
            Then, the closed loop dynamics \eqref{closed:dynamics_pflemma1:1}-\eqref{closed:dynamics_pflemma1:3} are equivalent to
            \begin{eqnarray}
             \label{dot_tilde_eta_pf_lemma}\dot{\tilde{\eta}}_i&=& -k_{\eta_i}\tilde{\eta}_i + \alpha k_{\eta_i}\bar{C}_i \tilde{\zeta}_i + (1-\alpha)k_{\eta_i}u_i + \mathcal{Y}_i\\
             \label{dot_tilde_zeta_pf_lemma}\dot{\tilde{\zeta}}_i &=& \bar{H}_i \tilde{\zeta}_i + \bar{B}_i u_i + \Upsilon_{i}
             \end{eqnarray}
             with  $\mathcal{Y}_{i}:= \Phi_{i,1} - \dot{\eta}_{c_i}$, and $\Upsilon_{i}:=  \Phi_{i,2} - \dot{\bar{\eta}}_{c_i}$ for $i\in\mathfrak{F}$.

            In view of \eqref{tilde_eta_delta}, the results of Lemma~\ref{theorem_unified} can be verified if the error vectors $\tilde{\eta}_i$, $i\in\mathfrak{F}$, and their first time-derivatives are uniformly bounded and converge asymptotically to zero. This can be shown using the result of Theorem~\ref{theorem001a} as follows. First, one needs to show that the overall system that consists of all systems \eqref{dot_tilde_eta_pf_lemma}-\eqref{dot_tilde_zeta_pf_lemma}, $i\in\mathfrak{F}$, is IOS with respect to some appropriately defined input and output vectors, and, in addition, the input vectors satisfy property \eqref{commconstraints0010} in Theorem~\ref{theorem001a}. Note that all systems \eqref{dot_tilde_eta_pf_lemma}-\eqref{dot_tilde_zeta_pf_lemma}, $i\in\mathfrak{F}$, are interconnected through the input vectors $u_i$, $i\in\mathfrak{F}$, given in \eqref{input_for_ISS}. Then, if one can obtain explicit expressions of the IOS gain matrix $\mathscr{G}^0$ and the interconnection matrix $\mathcal{M}$ defined in Theorem~\ref{theorem001a}, our objective can be shown under the conditions of Theorem~\ref{theorem001a}; $\bs{\rho}(\mathscr{G})<1$ where $\mathscr{G} = \mathscr{G}^0 \cdot\mathcal{M}$ is the closed loop gain matrix.

            Consider each system \eqref{dot_tilde_eta_pf_lemma}-\eqref{dot_tilde_zeta_pf_lemma} for $i\in\mathfrak{F}$, and let
             $\tilde{\zeta}_i:=\big(\tilde{\zeta}_{i,1}^{\top}, \ldots, \tilde{\zeta}_{i,\sigma_i}^{\top}\big)^{\top}\in\mathbb{R}^{\sigma_iN}$ and $\Upsilon_{i}:=\big(\Upsilon_{i,1}^{\top}, \ldots, \Upsilon_{i,\sigma_i}^{\top}\big)^{\top}\in\mathbb{R}^{\sigma_iN}$, with ${\tilde \zeta}_{i,\ell}\in\mathbb{R}^N$ and $\Upsilon_{i,\ell}\in\mathbb{R}^N$, for $\ell = 1, \ldots, \sigma_i$ and $i\in\mathfrak{F}$. Exploiting the structure of $\bar{H}_i$, $\bar{B}_i$, and $\bar{C}_i$ in \eqref{gains_filter},  one can show that the following estimates
             {\small{\begin{align}\label{estimate_tilde_eta_pf_lemma}
                 	\left|\tilde{\eta}_i(t)\right|
                  	\le&~
                          	e^{-k_{\eta_i}(t-t_0)}\left|\tilde{\eta}_i(t_0)\right|
                      	+ \alpha\sup\limits_{\varsigma\in[t_0, t]}\big(| \tilde{\zeta}_{i,1}(\varsigma)|\big)\nonumber \\
                      &+ (1-\alpha)\sup\limits_{\varsigma\in[t_0, t]}\big(|u_i(\varsigma)|\big)
                          +\frac{1}{k_{\eta_i}}\sup\limits_{\varsigma\in[t_0, t]}\big(|\mathcal{Y}_{i}(\varsigma)|\big),
            	\end{align}}}
              {\small{\begin{align}
                              	\left|
                              	\tilde{\zeta}_{i,\ell}(t) \right|
                      	\le&~e^{-h_{i,\ell}(t-t_0)}\left|
                              	\tilde{\zeta}_{i,\ell}(t_0)\right|
                          	+\sup\limits_{\varsigma\in[t_0, t]}\big(| \tilde{\zeta}_{i,\ell+1}(\varsigma)|\big)\nonumber\\
                          &+\frac{1}{h_{i,\ell}}\sup\limits_{\varsigma\in[t_0, t]}\big(|\Upsilon_{i,\ell}(\varsigma)|\big),
                	\end{align}}}
         for $\ell=1,\ldots, \sigma_i-1$, and
         {\small{\begin{align}\label{estimate_tilde_zeta_2_pf_lemma}
                              	\left|
                              	\tilde{\zeta}_{i,\sigma_i}(t) \right|
                      	\le&~
                              	e^{-h_{i,\sigma_i}(t-t_0)}\left|
                              	\tilde{\zeta}_{i,\sigma_i}(t_0)\right|
                          	+
                          \sup\limits_{\varsigma\in[t_0, t]}\big(|u_{i}(\varsigma)|\big) \nonumber\\
                           &+ \frac{1}{h_{i,\sigma_i}}\sup\limits_{\varsigma\in[t_0, t]}\big(| \Upsilon_{i,\sigma_i}(\varsigma)|\big)
                	\end{align}}}
         hold for all $t>t_0 \geq 0$ and $i\in\mathfrak{F}$. The above inequalities show that each system \eqref{dot_tilde_eta_pf_lemma}-\eqref{dot_tilde_zeta_pf_lemma}, $i\in\mathfrak{F}$, is a cascade connection of $\sigma_i + 1$ subsystems, where each subsystem is ISS. Therefore, each system \eqref{dot_tilde_eta_pf_lemma}-\eqref{dot_tilde_zeta_pf_lemma}, for $i\in\mathfrak{F}$, is also ISS with respect to the inputs $u_i$, $\mathcal{Y}_i$, and the components of $\Upsilon_i$. Consequently, the system \eqref{dot_tilde_eta_pf_lemma}-\eqref{dot_tilde_zeta_pf_lemma} with output $\tilde{\eta}_i$ can be shown to be IOS with respect to the same input vectors, in particular, the IOS gain with respect to the input $u_i$ is equal to 1 for both values of $\alpha$. As a result, if one considers the overall system that consists of all the systems \eqref{dot_tilde_eta_pf_lemma}-\eqref{dot_tilde_zeta_pf_lemma}, $i\in\mathfrak{F}$, with the $m$ outputs $\tilde{\eta}_i$ and the $m$ inputs $u_i$, $i\in\mathfrak{F}$, one can verify that thus defined system is IOS with the IOS gain matrix $\mathscr{G}^0= I_m$. Note that $\mathscr{G}^0$ does not relate the outputs $\tilde{\eta}_i$ to the inputs $\mathcal{Y}_i$ and $\Upsilon_i$, $i\in\mathfrak{F}$, which does not affect our analysis. In fact, one can show, using \begin{equation}\label{dot_eta_c}
            \dot{\bs{\eta}}_c = (-\mathcal{L}_1^{-1}\mathcal{L}_2\otimes I_N) \dot{\bs{\eta}}_L,
            \end{equation}
         that $\mathcal{Y}_i$ and $\Upsilon_i$, $i\in\mathfrak{F}$, are uniformly bounded and converge asymptotically to zero if $\Phi_{i,1}$, $\Phi_{i,2}$, $i\in\mathfrak{F}$, and $\Psi_{i}$, $i\in\mathfrak{L}$, are uniformly bounded and converge to zero.

         Now, we derive estimates of the inputs $u_i$, $i\in\mathfrak{F}$, given in \eqref{input_for_ISS}. In view of \eqref{eta_ci_pf_lemma} and the fact that $(t-t_{k_{ij}^{\mathrm{m}}(t)})\leq T^*$ (by Assumption~\ref{AssumptionCommunicationBlackouts}), it can be verified that
          \begin{align}\label{estimate_u_pf_lemma}
         |u_i(t)| \leq&\int_{t_{k_{ij}^{\mathrm{m}}(t)}}^{t}\Big|\frac{1}{\kappa_i}\sum_{j=1}^m a_{ij}\dot{\eta}_{c_j}(\varsigma) +\frac{1}{\kappa_i}
            \sum_{j=m+1}^n a_{ij}\dot{\eta}_{j}(\varsigma)\Big|d\varsigma,\nonumber\\
            &+ \frac{1}{\kappa_i}\sum_{j=1}^m a_{ij}\sup\limits_{\varsigma\in[t_{k_{ij}^{\mathrm{m}}(t)}, t]}\big(|\tilde{\eta}_j(\varsigma)|\big) \nonumber\\
          \leq&~ \frac{1}{\kappa_i}\sum_{j=1}^m a_{ij}\sup\limits_{\varsigma\in[t_{k_{ij}^{\mathrm{m}}(t)}, t]}\big(|\tilde{\eta}_j(\varsigma)|\big) + |\Delta_i|,
         \end{align}
         with $|\Delta_i| := T^*\sup\limits_{\varsigma\in[t_{k_{ij}^{\mathrm{m}}(t)}, t]}\big(|\dot{\eta}_{c_i}(\varsigma)|\big)$.

          Then, the input vector $u_i$, $i\in\mathfrak{F}$, of system \eqref{dot_tilde_eta_pf_lemma}-\eqref{dot_tilde_zeta_pf_lemma} satisfies the conditions of Theorem~\ref{theorem001a} with the elements of the interconnection matrix ${\mathcal M}:=  \left\{{\mu}_{ij}\right\}\in\mathbb{R}^{m\times m}$ being obtained as $\mu_{ij}= \frac{a_{ij}}{\kappa_i}$ for $i,~j\in\mathfrak{F}$. This, with \eqref{D}-\eqref{laplacian}, lead to the closed-loop gain matrix
            \begin{equation}\mathscr{G}:=\mathscr{G}^0~ {\mathcal M}= \mathcal{D}_1^{-1}\mathcal{A}_1.
                      	   	\end{equation}
            Using the definition of $\mathcal{L}_1:= \mathcal{D}_1 - \mathcal{A}_1$, we know that $\mathcal{D}_1^{-1}\mathcal{L}_1 = I_m - \mathscr{G}$. Since $\mathcal{L}_1$ is a non-singular M-matrix, by Assumption~\ref{assumption_graph} and Lemma~\ref{lemma_Mei}, and $\mathcal{D}_1^{-1}$ is a diagonal matrix with strictly positive diagonal entries, it is straightforward to verify that  $(I_m - \mathscr{G})$ is also a non-singular M-matrix (see Definition~\ref{definition_M_matrix}) and hence $\bs{\rho}(\mathscr{G})<1$, which satisfies the condition in Theorem~\ref{theorem001a}.

            Note that for all $(i,j)\in\mathcal{E}$, Assumption~\ref{AssumptionCommunicationBlackouts} ensures that $(t-t_{k_{ij}^{\mathrm{m}}(t)})$ is bounded and satisfies $(t-t_{k_{ij}^{\mathrm{m}}(t)})\leq T^*$. Then, in view of the IOS property of each system \eqref{dot_tilde_eta_pf_lemma}-\eqref{dot_tilde_zeta_pf_lemma}, $i,\in\mathfrak{F}$, Theorem~\ref{theorem001a} can be used to show that $\tilde{\eta}_i$, $\dot{\tilde{\eta}}_i$, $\tilde{\zeta}_i$, $\dot{\tilde{\zeta}}_i$, and $u_i$, for $i\in\mathfrak{F}$, are uniformly bounded provided that $|\Delta_i|$, with $\mathcal{Y}_i$ and $\Upsilon_i$, $i\in\mathfrak{F}$, are uniformly bounded. We can verify that the latter conditions are satisfied under the conditions of item $i)$ in Lemma~\ref{theorem_unified}. In particular, relation \eqref{dot_eta_c} with Assumption~\ref{assumption_graph} imply that $|\Delta_i|$, $i\in\mathfrak{F}$, is uniformly bounded if $\dot{\eta}_i:= \Psi_{i}$, $i\in\mathfrak{L}$, is uniformly bounded. As a result, we conclude that $\tilde{\eta}_i$, $\dot{\tilde{\eta}}_i$, $\tilde{\zeta}_i$, $\dot{\tilde{\zeta}}_i$, for $i\in\mathfrak{F}$ are uniformly bounded. Also, it can be verified from \eqref{tilde_eta_delta} and \eqref{dot_eta_c} that $\dot{\eta}_i$ and $\dot{\zeta}_i$, $i\in\mathfrak{F}$, are uniformly bounded. This proves statement $i)$ in the lemma.

            Using similar arguments as above, we can show that $|\Delta_i(t)|\to 0$, $i\in\mathfrak{F}$, under the conditions of item $ii)$ in Lemma~\ref{theorem_unified}. Also, $\mathcal{Y}_i(t)\to 0$ and $\Upsilon_i(t)\to 0$, $i\in\mathfrak{F}$, under the same conditions. Then, Theorem~\ref{theorem001a}, with the IOS property of systems \eqref{dot_tilde_eta_pf_lemma}-\eqref{dot_tilde_zeta_pf_lemma}, can be used to show that $\tilde{\eta}_i(t)\to 0$, $\dot{\eta}_i(t)\to 0$, $\tilde{\zeta}_i(t)\to 0$, and $\dot{\zeta}_i(t)\to 0$, $i\in\mathfrak{F}$, which leads to the conclusions in item~$ii)$ in the lemma.

            Finally, let us consider the case where $\Phi_{i,1}$, $\Phi_{i,2}$, $i\in\mathfrak{F}$, and $\Psi_{i}$, $i\in\mathfrak{L}$, are uniformly bounded but do not converge to zero. We can deduce from \eqref{estimate_tilde_eta_pf_lemma}-\eqref{estimate_tilde_zeta_2_pf_lemma} and \eqref{dot_eta_c}-\eqref{estimate_u_pf_lemma} that the effects of non-zero perturbation terms $\Phi_{i,1}$, $\Phi_{i,2}$, $i\in\mathfrak{F}$, and $\Psi_{i}$, $i\in\mathfrak{L}$,  on the states of system \eqref{dot_tilde_eta_pf_lemma}-\eqref{dot_tilde_zeta_pf_lemma} can be reduced with a choice of the gains provided that $T^*$ is small and/or $\bs{\eta}_L$ is a slowly varying signal. In fact, it can be verified, form \eqref{estimate_tilde_eta_pf_lemma}-\eqref{estimate_tilde_zeta_2_pf_lemma}, that the effects of
            $\Upsilon_i : =\Phi_{i,2} - (\mbf{1}_{\sigma_i}\otimes I_N)\dot{\eta}_{c_i}$ (in the case $\alpha = 1$) and $\mathcal{Y}_i:=\Phi_{i,1}-\dot{\eta}_{c_i}$ can be {\it arbitrarily} reduced, respectively, with a choice of the entries of the gain matrix $\bar{H}_i$ and the gains $k_{\eta_i}$. In addition, it can be verified from \eqref{estimate_u_pf_lemma} that $|\Delta_i(t)|$, $i\in\mathfrak{F}$ is small for small values of $T^*$ and/or small values of $|\dot{\eta}_{c_i}|$, $i\in\mathfrak{F}$, where, in view of \eqref{dot_eta_c},  $|\dot{\eta}_{c_i}|$ depends on $|\dot{\eta}_i| = |\Psi_i|$, $i\in\mathfrak{L}$. The proof is complete.


\end{document}